\documentclass{patmorin}
\usepackage{pat}

\usepackage[T1]{fontenc}
\usepackage[utf8]{inputenc}
\usepackage{lmodern}
\usepackage[UKenglish]{babel}
\usepackage[margin=1in]{geometry}
\usepackage{amsmath,amssymb,amsthm}
\usepackage{graphicx}
\usepackage[dvipsnames]{xcolor}
\usepackage{enumitem} 
\usepackage{calc} %
\usepackage{microtype}
\usepackage[numbers,sort&compress]{natbib}
\usepackage{pifont}
\usepackage{stmaryrd}
\definecolor{brightmaroon}{rgb}{0.76, 0.13, 0.28}
\newcommand{\defin}[1]{\emph{\textcolor{brightmaroon}{#1}}}

\crefname{thm}{Theorem}{Theorems}
\crefname{lem}{Lemma}{Lemmas}
\crefname{obs}{Observation}{Observations}
\crefname{cor}{Corollary}{Corollaries}
\crefname{clm}{Claim}{Claims}
\Crefname{thm}{Theorem}{Theorems}
\Crefname{lem}{Lemma}{Lemmas}
\Crefname{obs}{Observation}{Observations}
\Crefname{cor}{Corollary}{Corollaries}
\Crefname{clm}{Claim}{Claims}

\title{\MakeUppercase{Connected Dominating Sets in Triangulations}\thanks{An extended abstract of these results appeared at the 2026 International Colloquium on Automata, Languages, and Programming (ICALP 2026).}}
\author{Prosenjit Bose\thanks{Email: \href{mailto:jit@scs.carleton.ca}{\texttt{jit@scs.carleton.ca}}. ORCID: \href{https://orcid.org/0000-0002-8906-0573}{0000-0002-8906-0573}. Research partially funded by NSERC.} \quad
Vida Dujmovi\'c\,\thanks{Email: \href{mailto:vida.dujmovic@uottawa.ca}{\texttt{vida.dujmovic@uottawa.ca}}. ORCID: \href{https://orcid.org/0000-0001-7250-0600}{0000-0001-7250-0600}. Research partially funded by NSERC and a University of Ottawa Research Chair.} \quad
Hussein Houdrouge\thanks{Email: \href{mailto:houdrouge.hussein@gmail.com}{\texttt{houdrouge.hussein@gmail.com}}. ORCID: \href{https://orcid.org/0009-0007-4908-5294}{0009-0007-4908-5294}. Research partially funded by NSERC.} \quad
Pat Morin\thanks{Email: \href{mailto:morin@scs.carleton.ca}{\texttt{morin@scs.carleton.ca}}. ORCID: \href{https://orcid.org/0000-0003-0471-4118}{0000-0003-0471-4118}. Research partially funded by NSERC.} \quad
Saeed Odak\thanks{Email: \href{mailto:saeedodak@gmail.com}{\texttt{saeedodak@gmail.com}}. ORCID: \href{https://orcid.org/0009-0005-6290-0965}{0009-0005-6290-0965}.}}

\date{}

\begin{document}

\maketitle

\begin{abstract}
A dominating set of a graph $G$ is \emph{connected} if it induces a connected graph in $G$. For planar triangulations, it has been known since 1990 that every $n$-vertex triangulation with $n\ge4$ admits a connected dominating set of size at most $n/2 - 1$, and no improvement to this bound was known for over three decades. We break this longstanding barrier by showing that every $n$-vertex triangulation has a connected dominating set of size at most $10n/21$. Equivalently, every triangulation admits a spanning tree with at least $11n/21$ leaves. 
This result has graph drawing applications, establishing a bound for one-bend free sets and improving the known bound for simultaneous planar embeddings. 
The connected dominating set results extend readily to show that an $n$-vertex triangulation of a surface of Euler genus $g$ has a connected dominating set of size at most $10n/21+O(\sqrt{gn})$.
\end{abstract}

{\small
\noindent\textbf{Keywords.} connected domination, triangulations, planar graphs, graph drawing, collinear sets
\par\smallskip
\par}
\medskip

\section{Introduction}
A set $X$ of vertices in a graph $G$ is a \defin{dominating set} of $G$ if each vertex of $G$ is in $X$ or adjacent to a vertex in $X$.\footnote{Any graph $G$ that we consider in this paper is a finite undirected graph without self-loops and without parallel edges, with vertex set $V(G)$ and edge set $E(G)$.}  A dominating set $X$ of $G$ is \defin{connected} if the subgraph $G[X]$ of $G$ induced by the vertices in $X$ is connected.  There is an enormous body of literature on dominating sets. Several books are devoted to the topic \cite{haynes.hedetniemi.ea:domination,haynes.hedetniemi.ea:topics,du.wan:connected,haynes.hedetniemi.ea:vol2}, including a book and book chapter devoted to connected dominating sets \cite{du.wan:connected,chellali.favaron:connected}. Connected dominating sets have numerous applications, particularly in areas such as wireless ad hoc networks, as well as broadcasting and multi-casting; see, for example, \cite{wu:dominating}.

A typical result in the area is an upper bound of the form: ``Every $n$-vertex graph in some family $\mathcal{G}$ of graphs has a (connected) dominating set of size at most $f(n)$,'' or a lower bound of the form ``for infinitely many $n$, there exists an $n$-vertex member of $\mathcal{G}$ with no (connected) dominating set of size less than $g(n)$.''

One family of graphs that has received considerable attention in this context is the class of triangulations, that is, edge-maximal planar graphs with at least three vertices. \citet{matheson.tarjan:dominating} proved that every $n$-vertex triangulation has a dominating set of size at most $n/3=0.33\overline{3}n$ and that, for infinitely many $n$, there exist $n$-vertex triangulations with no dominating set of size less than $n/4=0.25n$. The gap between these upper and lower bounds stood for over $20$ years until a recent breakthrough reduced the upper bound to $17n/53\approx 0.3208n$ for $n>6$ \cite{spacapan:domination}.  This was swiftly followed by an improvement to $2n/7= 0.2\overline{857142}n$ for $n>10$  \cite{christiansen.rotenberg.ea:triangulations}.

The focus of the current paper is on the existence of small connected dominating sets in triangulations. In the following sections, we describe the history of this problem and then present our main results. At the end of the paper, we present corollaries of our results for two applications in graph drawing. 

Connected dominating sets are complementary to the leaves of spanning trees in the following sense: a set of vertices $X$ is a connected dominating set of an $n$-vertex graph $G$ with $n\ge3$ if and only if $G$ has a spanning tree in which the vertices of $G-X$ are all leaves in that tree. Thus, connected dominating sets are studied implicitly in the literature on spanning trees with many leaves \cite{DBLP:journals/jgt/DingJS01, kleitman.west:spanning, doi:10.1137/100801251,doi:10.1137/060664318, DBLP:conf/esa/BousquetI0MOSW20, noguchi.zamfirescu:spanning}. A closely related and extensively studied notion is that of spanning trees with no degree-$2$ vertices, known as \emph{homeomorphically irreducible spanning trees} \cite{chen.ren.ea:homeomorphically,albertson.berman.ea:graphs,CHEN2013409, DBLP:journals/jgt/ItoT22, DBLP:journals/jgt/Hoffmann-Ostenhof18a, gessel2023good}.  Since an $n$-vertex tree with $n\ge3$ and no degree-$2$ vertices has at least $n/2+1$ leaves, the existence of a homeomorphically irreducible spanning tree implies the existence of a spanning tree with at least $n/2+1$ leaves (and the existence of a connected dominating set of size at most $n/2-1$). 
As stated already, we consider connected dominating sets in triangulations.  An easy consequence of the proof in \cite{matheson.tarjan:dominating} is that $n$-vertex triangulations have connected dominating sets of size at most $2n/3=0.66\overline{6}n$.  A more general result \cite{kleitman.west:spanning} shows that connected graphs of minimum degree at least $3$ in which each edge is included in a $3$-cycle have connected dominating sets of size at most $(2n-5)/3<0.66\overline{6}n$. \citet{albertson.berman.ea:graphs} prove that every triangulation with at least four vertices has a homeomorphically irreducible spanning tree which, as discussed above, implies that every such triangulation has a connected dominating set of size at most $n/2-1<0.5n$.

\citet{chen.ren.ea:homeomorphically} give a significant generalization of this result, which applies to any connected graph $G$ with more than three vertices in which the graph induced by the neighbours of each vertex is connected.
Answering an open problem posed by \citet{albertson.berman.ea:graphs} and confirming the conjecture of Archdeacon \cite{Gross_Tucker_2009}, \citet{CHEN2013409} generalizes this result to connected graphs in which every edge is in at least two $3$-cycles.

Motivated by the fact that the $n/2$ upper bound for triangulations has stood for over three decades, several authors have asked if it can be improved. \citet[Question~4.2]{bradshaw.masarik.ea:robust} ask if it can be improved to $n/3$. \citet{bryant.pavelescu:connected} computed connected domination numbers for all plane triangulations with at most thirteen vertices and also asked whether every plane triangulation has a connected dominating set of size at most $n/3$. As it turns out, the answer is negative (see below for more on the lower bound).  \citet{noguchi.zamfirescu:spanning} posed an even more modest question, asking whether the bound could be improved to $(\tfrac{1}{2}-\epsilon)n$ for some $\epsilon>0$, even when restricted to the special case of $4$-connected triangulations.

Our main contribution is the first improvement to the longstanding $n/2$ upper bound for connected dominating sets in triangulations, as stated in the following theorem.

\begin{thm}\label{main_result2}
  For every $n\ge 3$, every $n$-vertex triangulation $G$ has a connected dominating set $X$ of size at most $10n/21= 0.\overline{476190}n$.  Equivalently, $G$ has a spanning tree $T$ with at least $11n/21=0.\overline{523809}n$ leaves.
\end{thm}

Until recently, the best known lower bound for this problem, illustrated in \cref{lower_bound}, was $n/3 = 0.33\overline{3} n$, obtained from a triangulation with $n=3k$, $k\ge4$, that contains $k$ vertex-disjoint pairwise-nested triangles $\Delta_1,\ldots,\Delta_k$. The annuli between consecutive triangles are triangulated so that each vertex has exactly two neighbours in each adjacent layer: writing $\Delta_i=a_i b_i c_i$, the cross edges are $a_i a_{i+1}$, $a_i c_{i+1}$, $b_i b_{i+1}$, $b_i a_{i+1}$, $c_i c_{i+1}$, and $c_i b_{i+1}$.  In order to dominate $\Delta_1\cup \Delta_2$, any connected dominating set must contain at least two vertices in $\Delta_1 \cup \Delta_2$. In order to dominate $\Delta_{k-1}\cup \Delta_k$, any connected dominating set must contain at least two vertices in $\Delta_{k-1}\cup\Delta_k$.  Then, in order to be connected, any connected dominating set must contain a vertex in each of $\Delta_3,\ldots,\Delta_{k-2}$. Conversely, $\{a_1,\ldots,a_k\}$ induces a path and dominates every layer, so the minimum size is exactly $k=n/3$.

\begin{figure}[htbp]
  \centering
  \includegraphics{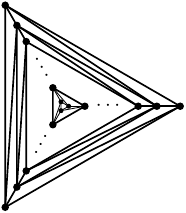}
  \caption{A triangulation with $n=3k$, $k\ge4$, vertices whose smallest connected dominating set has size $n/3$.}
  \label{lower_bound}
\end{figure}

Since this example contains many separating triangles, it is natural to consider the special case of $4$-connected triangulations.  \citet{noguchi.zamfirescu:spanning} describe, for infinitely many values of $n$, $4$-connected $n$-vertex triangulations for which any connected dominating set has at least $n/3$ vertices.  

The lower bound of $n/3$ for arbitrary triangulations stood for many years until very recently:  \citet{enami2025contributionsconjecturesplanargraphs} gave examples of $n$-vertex triangulations in which every connected dominating set has size at least $\frac{7}{18}n - \frac{2}{3}$ (where $7/18 \approx 0.388\overline{8}$). 

A \defin{surface triangulation} is a cellular embedding of a finite simple
graph on a closed connected surface, with each face bounded by a $3$-cycle.
For Euler genus $g$, we extend \cref{main_result2} to connected dominating sets 
of size $10n/21+O(\sqrt{gn})$ and spanning trees with at least
$11n/21-O(\sqrt{gn})$ leaves; see \cref{genus_result}.

\subsection{Outline}

\paragraph{Overview of the Approach.}
For a plane graph $G$, let $B(G)$ be the set of vertices of $G$ on the outer face of $G$ (the \defin{boundary} vertices) and let $I(G):=V(G)\setminus B(G)$ (the \defin{inner} vertices).
A natural strategy for computing a small connected dominating set in a triangulation is \defin{greedy peeling}: Repeatedly choose a vertex $v\in B(G)$ that maximizes $|N_G(v)\cap I(G)|$, remove $v$ from the graph and add $v$ to the dominating set $X$. In a triangulation, the first vertex $v_0$ dominates all vertices in $B(G-v_0)$ and every subsequent step ensures that all vertices on the boundary of the remaining graph are dominated by $X$ and every vertex in $X$ has a path to $v_0$ in $G[X]$. The process ends when all vertices of the current graph are boundary vertices.  A fundamental barrier to improving the $n/2$ bound with greedy peeling is the occurrence of boundary vertices of inner degree~$2$: deleting such a vertex adds one vertex to $X$ and only increases the number of boundary vertices by $1$ (because $v$ disappears from the boundary), so this process requires roughly $(n-3)/2$ steps. Our key insight is that the algorithm is never forced to accept an unfavourable single degree-$2$ step in isolation:
\begin{enumerate}[label=\rm(\arabic*),nosep,nolistsep]
  \item It is sometimes possible to ensure that a vertex $v$ with inner degree $2$ can immediately be followed by a vertex $w$ with inner degree at least $3$ (a $2$--$3$ combo), which reduces the number of inner vertices by $5$ while increasing the size of $X$ by $2$.
  \item If $G - B(G)$ has low-degree vertices ($0$ or $1$), a local pruning strictly decreases the boundary while eliminating inner vertices.
  \item If none of these apply, the remaining graph has a very special structure, which we call \defin{$2$-criticality}, and the remaining vertices can be chosen using a global strategy based on the fact that outerplanar graphs are $3$-colourable.
\end{enumerate}
We track the progress of this strategy using a potential function $\Phi$ that begins with a value of at most $10n/21$, decreases by at least $1$ for each vertex added to $X$ at all but the final step, and whose value immediately before the final step is at least the number of vertices added in the final step.

\paragraph{Paper Organization.}
The remainder of this paper is organized as follows:  In \cref{strategy}, we describe the general strategy we use for finding connected dominating sets in triangulations.  In \cref{warm_up} we show that a simple version of this strategy can be used to obtain a connected dominating set of size at most $4n/7= 0.\overline{571428}n$. In \cref{full_result} we show that a more careful construction leads to a proof of \cref{main_result2}. In \cref{bounded_genus}, we extend the result to surface triangulations. In \cref{gd}, we discuss the connection between connected dominating sets and two graph drawing applications. Finally, \cref{conclusions} concludes by pointing out directions for future work.

\section{The General Strategy}
\label{strategy}

Throughout this paper, we use standard graph-theoretic terminology as used, for example, by \cite{diestel:graph}. All graphs are finite and simple.
For a graph $G$, let $|G|=|V(G)|$ denote the number of vertices of $G$.  A \defin{bridge} in a graph $G$ is an edge $e$ of $G$ such that $G-e$ has more connected components than $G$.  For a vertex $v\in V(G)$, $N_G(v):=\{w\in V(G):vw\in E(G)\}$ is the \defin{open neighbourhood} of $v$ in $G$, and $N_G[v]:=N_G(v)\cup\{v\}$ is the \defin{closed neighbourhood} of $v$ in $G$.  For a vertex subset $S\subseteq V(G)$, $N_G[S]:=\bigcup_{v\in S} N_{G}[v]$ is the \defin{closed neighbourhood} of $S$ in $G$ and $N_G(S):=N_G[S]\setminus S$ is the \defin{open neighbourhood} of $S$ in $G$.  A set $X\subseteq V(G)$ \defin{dominates} a set $B\subseteq V(G)$ if $B\subseteq N_G[X]$.  Thus, $X$ is a dominating set of $G$ if and only if $X$ dominates $V(G)$.

A \defin{plane graph} is a graph equipped with a non-crossing embedding in $\mathbb{R}^2$. Subgraphs inherit this embedding.  A plane graph is \defin{outerplane} if all its vertices appear on the outer face.  A \defin{triangle} is a cycle of length $3$. A \defin{near-triangulation} is a plane graph whose outer face is bounded by a cycle and whose inner faces are all bounded by triangles.  A \defin{generalized near-triangulation} is a plane graph whose inner faces are bounded by triangles. A generalized near-triangulation may have multiple components, cut vertices, and bridges, and may be the empty graph. Since every bounded face has a single triangular boundary, all cut vertices and bridges lie on the outer face of a generalized near-triangulation, and each biconnected component with at least three vertices is embedded as a near-triangulation.

In several places, we will make use of the following observation, which is really a statement about the triangulation contained in a cycle of length $3$.
\begin{obs}\label{useful_little_guy}
  Let $H$ be a generalized near-triangulation and let $xyz$ be a cycle in $H$.  Then,
  \begin{enumerate}[label=\rm(\alph*),nosep,nolistsep]
    \item\label{useful_little_guy:1} If the interior of $xyz$ contains at least one vertex of $H$, then each of $x$, $y$, and $z$ has at least one neighbour in the interior of $xyz$.

    \item\label{useful_little_guy:2} If the interior of $xyz$ contains at least two vertices of $H$, then at least two of $x$, $y$, and $z$ have at least two neighbours in the interior of $xyz$.
  \end{enumerate}
\end{obs}
\begin{proof}
If $x$ has no neighbour inside $xyz$, the triangular face on the inside of
$xy$ is $xyz$, so the triangle has no interior vertex. This proves the first
claim, and the same argument applies to $y$ and $z$.
For the second claim, suppose that two corners, say $x$ and $y$, each have
only one interior neighbour. The triangular face on the inside of $xy$
shows that this neighbour is the same vertex $w$. The faces on the inside
of $xz$ and $yz$ must then be $xzw$ and $yzw$. Together these three faces
fill the triangle, leaving $w$ as its only interior vertex, a contradiction.
\end{proof}

For a plane graph $H$, we use the notation $B(H)$ to denote the vertex set of the outer face of $H$ and define $I(H):=V(H)\setminus B(H)$.  The vertices in $B(H)$ are \defin{boundary vertices} of $H$ and the vertices in $I(H)$ are \defin{inner vertices} of $H$. For any vertex $v$ of $H$, the \defin{inner neighbourhood} of $v$ in $H$ is defined as $N_H^+(v):=N_H(v)\cap I(H)$, the vertices in $N^+_H(v)$ are \defin{inner neighbours} of $v$ in $H$, and $\deg^+_H(v)=|N^+_H(v)|$ is the \defin{inner degree} of $v$ in $H$.

Let $G$ be a triangulation.  Our procedure for constructing a connected dominating set $X$ begins with an incremental phase that eats away at $G$ ``from the outside.'' The process of constructing $X$ is captured by the following definition: Let $G_0:=G$, let $X\subseteq V(G)$ and let $\Delta_0,\ldots,\Delta_{r}$ be a partition of $X$. For each $i\in\{1,\ldots,r+1\}$, let $X_i = \bigcup_{j=0}^{i-1}\Delta_j$, and for each $i\in\{0,\ldots,r\}$, let $G_i:=G-X_i$. We say that $\Delta_0,\ldots,\Delta_{r}$ is an \defin{outer-domatic partition} of $X$ if
\begin{enumerate}[label=(P\arabic*),labelwidth=\widthof{(P2)}, leftmargin=!]
  \item\label{back_connected} each component of $G_i[\Delta_i]$ contains a vertex in $B(G_i)$, for each $i\in\{0,\ldots,r\}$; and
  \item $X_{r+1}$ dominates $G$.\label{final_domination}
\end{enumerate} 
We say that $X$ is \defin{outer-domatic} if $X$ has an outer-domatic partition.

\begin{lem}\label{outer_domatic}
    Let $G$ be a triangulation.  Then any outer-domatic $X\subseteq V(G)$ is a connected dominating set of $G$.
\end{lem}
\begin{proof}
  Suppose $X$ is outer-domatic, let $\Delta_0,\ldots,\Delta_{r}$ be an outer-domatic partition of $X$, and let $X_1,\ldots,X_{r+1}$ and $G_0,\ldots,G_r$ be as in the definition of outer-domatic partition.  By \ref{final_domination}, $X_{r+1}=X$ is a dominating set of $G$. We need only show that $G[X]$ is connected.  
  For each $i\in\{1,\ldots,r\}$, each vertex $v\in B(G-X_i)$ is either in $B(G)$ or adjacent to a vertex in $X_i$. 
  Thus, for each $i\in\{0,\ldots,r\}$ and each $v\in \Delta_i$, the component of $G_i[\Delta_i]$ containing $v$ contains a vertex $v'\in B(G_i)$, by \ref{back_connected}.  If $i\ge 1$, then $v'$ is adjacent to a vertex $u\in X_{i}$.  By definition, $u\in \Delta_j$ for some $j\in\{0,\ldots,i-1\}$.  Thus, $G[X]$ contains a path from $v$ to a vertex in $\Delta_0\cap B(G)$.  Since $G[B(G)]$ is a clique, this implies that $G[X]$ is connected.
\end{proof}

We will present two algorithms that grow a connected dominating set in small batches $\Delta_0,\Delta_1,\ldots,\Delta_{r-1}$ followed by a terminal batch $\Delta_r$, resulting in a sequence of sets $X_1,\ldots,X_{r+1}$ where $X_{i}=\bigcup_{j=0}^{i-1}\Delta_j$.  The simple algorithm is unable to continue with single-vertex greedy steps once it reaches a point where each vertex in $B(G-X_i)$ has inner degree at most $1$ in $G-X_i$.  We begin by studying the graphs that cause this to happen.

\subsection{Critical Graphs}

A generalized near-triangulation $H$ is \defin{critical} if $\deg^+_H(v)\le 1$ for each $v\in B(H)$. We say that an inner face of $H[B(H)]$ is \defin{marked} if it contains an inner vertex of $H$.

\begin{figure}[htbp]
    \centering
    \includegraphics[page=1]{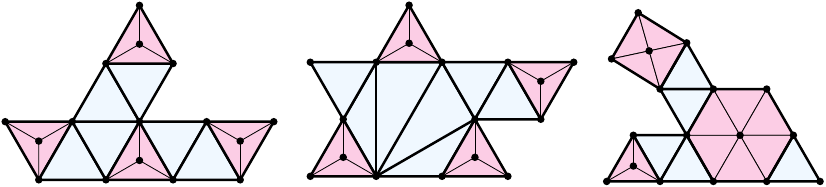}
    \caption{Some critical graphs $H$ with the marked faces of $H[B(H)]$ shaded.}
    \label{critical_fig}
\end{figure}

\begin{lem}\label{critical_structure}
    Let $H$ be a critical generalized near-triangulation. Then each face $f$ of $H[B(H)]$ contains at most one vertex of $I(H)$. If it contains such a vertex, that vertex is adjacent to every vertex on the boundary of $f$.
\end{lem}

\begin{proof}
  Let $B:=B(H)$ and $I:=I(H)$.
  The graph $H[B]$ is outerplane. Let $f$ be a marked face. Since $H$
  triangulates $f$, choose an edge $vx$ with $v$ on its boundary and $x\in I$ inside $f$.
  List the distinct neighbours of $x$ in cyclic order as
  $v=v_0,v_1,\ldots,v_{k-1}$. Since $x$ is inner, every
  $xv_iv_{i+1}$ is a bounded triangular face, with indices modulo $k$.
  Starting with $v_0\in B$, criticality implies successively that every
  $v_i$ belongs to $B$: if $v_i\in B$ but $v_{i+1}\in I$, then
  $v_i$ would have the two inner neighbours $x$ and $v_{i+1}$.
  These facial triangles fill the disc bounded by
  $v_0,\ldots,v_{k-1}$, which is therefore the boundary of $f$.
  Thus $x$ is the unique inner vertex in $f$ and is adjacent to every
  vertex on its boundary.
\end{proof}

\begin{lem}\label{base_case}
    Let $H$ be a critical generalized near-triangulation. Then $|B(H)|\ge 3|I(H)|$ and there exists $\Delta\subseteq B(H)$ of size at most $|I(H)|$ that dominates $I(H)$.
\end{lem}
\begin{proof}
  Let $B:=B(H)$ and $I:=I(H)$.  If $I$ is empty, then the result is trivially true, by taking $\Delta:=\emptyset$, so we now assume that $I$ is nonempty.  By \cref{critical_structure}, $H$ is formed from the outerplane graph $H[B]$ by adding $|I|$ stars, one in the interior of each marked face of $H[B]$.  Since $\deg_H^+(v) \le 1$ for each $v\in B$, each vertex of $H[B]$ is on the boundary of at most one marked face.  For each vertex $w\in I$, the marked face $f$ of $H[B]$ that contains $w$ has at least $3$ vertices, which do not belong to any other marked face. Therefore $|B|\ge 3|I|$ and by choosing one vertex from each marked face of $H[B]$ we obtain the desired set $\Delta$.
\end{proof}
\subsection{A Simple Algorithm} \label{warm_up}
We start with a variant of the simplest possible greedy algorithm, which we call $\textsc{SimpleGreedy}(G)$, to choose $\Delta_0,\ldots,\Delta_{r}$. If $G=K_3$, we return any single vertex; otherwise, we proceed as follows.  Suppose we have already chosen $\Delta_0,\ldots,\Delta_{i-1}$ for some $i\ge 0$ and we now want to choose $\Delta_i$. For $i \geq 0$, we let $X_i:=\bigcup_{j=0}^{i-1}\Delta_j$, let $G_i:=G-X_i$, and let $v_i$ be a vertex in $B(G_i)$ that maximizes $\deg^+_{G_i}(v_i)$.  During iteration $i\ge 0$, there are only two cases to consider:
\begin{enumerate}[label={[g\arabic*]}]
    \item If $\deg^+_{G_i}(v_i)\ge 2$, then we set $\Delta_i := \{v_i\}$.
    \item If $\deg^+_{G_i}(v_i)\le 1$, then $G_i$ is critical and this is the final step, so $r:=i$.  By \cref{base_case}, there exists $\Delta_r\subseteq B(G_r)$ of size at most $|I(G_r)|$ that dominates $I(G_r)$. Then $X_{r+1}:=X_r\cup\Delta_r$ and we are done.
\end{enumerate}

\begin{thm}\label{simple_greedy}
  When applied to an $n$-vertex triangulation $G$ with $n\ge4$, $\textnormal{\textsc{SimpleGreedy}}(G)$ produces a connected dominating set $X_{r+1}$ of size at most $(4n-9)/7$.
\end{thm}

\begin{proof}
By the choice of $\Delta_0,\ldots,\Delta_{r}$, each component of $G_i[\Delta_i]$ contains a vertex in $B(G_i)$ (for $i < r$, $\Delta_i=\{v_i\}\subseteq B(G_i)$, and for $i=r$, $\Delta_r\subseteq B(G_r)$). Every vertex in $B(G_r)$ is dominated by $X_r$ (since either it belongs to $B(G)$, which is dominated by $\Delta_0\cap B(G)$, or it entered the boundary upon deleting an adjacent vertex in $X_r$), and every vertex in $I(G_r)$ is dominated by $\Delta_r$ by \cref{base_case}. Thus, $X_{r+1}$ is an outer-domatic subset of $V(G)$, so, by \cref{outer_domatic}, $X_{r+1}$ is a connected dominating set of $G$. For each $i\in\{0,\ldots,r\}$, let $I_i:=I(G_i)$ and $B_i:=B(G_i)$.  We bound the size of $X_{r+1}$ using the potential function
\[
   \Phi_i := \frac{4|I_i|+|B_i|}{7} \enspace .
\]
For each greedy step $i\in\{0,\ldots,r-1\}$, case~[g1] applies, so $\Delta_i=\{v_i\}$ and $G_{i+1}=G_i-v_i$. Let $t_i:=\deg^+_{G_i}(v_i)\ge 2$. Then $I(G_i-v_i)=I_i \setminus N^+_{G_i}(v_i)$ and $B(G_i-v_i) = (B_i \setminus \{v_i\}) \cup N^+_{G_i}(v_i)$.  Thus, $|I_{i+1}|=|I_i|-t_i$ and $|B_{i+1}|=|B_i|+t_i-1$. Then, for $i\in\{0,\ldots,r-1\}$, we have 
\[
   \Phi_i - \Phi_{i+1} = \frac{4(|I_{i}|-|I_{i+1}|)+(|B_{i}|-|B_{i+1}|)}{7}
   = \frac{4(t_i)-(t_i-1)}{7}=\frac{3t_i+1}{7} \ge 1 \enspace ,
\]
since $t_i\ge 2$.
Thus, 
\[
   \Phi_0-\Phi_{r} = \sum_{i=0}^{r-1} (\Phi_i - \Phi_{i+1}) \ge r = |X_r|\enspace .
\]
Since $G_0=G$ is an $n$-vertex triangulation with $n\ge 4$, its outer face is a $3$-cycle, so $|B_0|=3$ and $|I_0|=n-3$. Thus,
\[
  \Phi_0 = \frac{4(n-3)+3}{7} = \frac{4n-9}{7} \enspace .
\]
When $i=r$, case~[g2] applies, so $G_{r}$ is critical. By \cref{base_case}, $|B_{r}|\ge 3|I_{r}|$ and $|\Delta_r|\le |I_r|$. Therefore,
\[
  |\Delta_{r}| \le |I_{r}| = \frac{4|I_{r}|+3|I_{r}|}{7} \le \frac{4|I_{r}|+|B_{r}|}{7} = \Phi_{r} \enspace .
\]
Therefore,
\[
    |X_{r+1}| = |X_r| + |\Delta_r| \le (\Phi_0 - \Phi_r) + \Phi_r = \Phi_0 = \frac{4n-9}{7} \enspace . \qedhere
\]
\end{proof}

\section{A Better Algorithm: Proof of \texorpdfstring{\cref{main_result2}}{the Main Theorem}}
\label{full_result}

Next we devise an algorithm that produces a smaller connected dominating set than what $\textsc{SimpleGreedy}(G)$ can guarantee.  This involves a more careful analysis of the cases in which \textsc{SimpleGreedy} is forced to take a vertex $v_i$ with $\deg^+_{G_i}(v_i)=2$.  We will show that in most cases, any time the algorithm is forced to choose a vertex $v$ that has inner degree $2$ in $G_i$, this can immediately be followed by choosing a vertex $w$ that has inner degree at least $3$ in $G_i-v$.  This is explained in \cref{combo}.

When this is no longer possible, the algorithm will be forced to directly handle a graph $G_i$ in which $\deg^+_{G_i}(v)\le 2$ for all $v\in B(G_i)$ and $G_i-B(G_i)$ is critical.  In \cref{two_critical_section}, we explain how this can be done using a set $\Delta_i$ whose size depends only on $G_i-B(G_i)$.  The results in \cref{two_critical_section} require that the graph $G_i-B(G_i)$ not have any vertices of degree less than $2$.  The steps required to eliminate degree-$1$ and degree-$0$ vertices from $G_i-B(G_i)$ are explained in \cref{one_kill_sec,zero_kill_sec}.

\subsection{Dom-Minimal Dom-Respecting Graphs}
We begin by identifying unnecessary vertices and edges that can appear in the graphs $G_1,\ldots,G_{r-1}$ during the construction of $X$. Refer to \cref{dom_minimal_fig}.   We say that a near-triangulation $H$ is \defin{dom-minimal} if
\begin{enumerate}[label=(DM\arabic*),labelwidth=\widthof{(DM0)},leftmargin=!]
    \item each vertex $v\in B(H)$ has $\deg^+_H(v)\ge 1$;  \label{bad_vertex}
    \item for each $v\in B(H)$ with $\deg^+_H(v)=1$, $H[N_H[v]]$ is isomorphic to $K_4$; and \label{inner_degree_1}
    \item each edge $vw$ on the boundary of the outer face of $H$ is incident to an inner face $vwx$ with $x\in I(H)$. \label{bad_edge}
\end{enumerate}
We say that a generalized near-triangulation $H$ is \defin{dom-minimal} if it has no isolated vertices, no bridges, and each of its biconnected components\footnote{A \defin{biconnected component} is a maximal $2$-connected induced subgraph, with at least three vertices.} is dom-minimal. The empty graph is dom-minimal.

\begin{figure}[htbp]
  \centering
  \includegraphics[page=2]{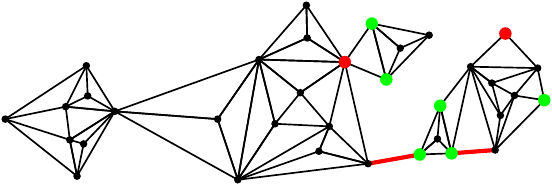}
  \caption{A generalized near-triangulation that is not dom-minimal. The \textcolor{red}{red} vertices violate \ref{bad_vertex}, the \textcolor{green}{green} vertex violates \ref{inner_degree_1}, and the \textcolor{red}{red} edges violate \ref{bad_edge}.}
  \label{dom_minimal_fig}
\end{figure}

Let $H$ and $H'$ be two generalized near-triangulations.  We say that $H'$ \defin{dom-respects} $H$ if
\begin{enumerate}[label=(DR\arabic*),labelwidth=\widthof{(DR0)},leftmargin=!]
  \item $B(H')\subseteq B(H)$; \label{boundary_subset}
  \item $H'[I(H')]=H[I(H)]$; and \label{same_inner_vertices}
  \item $N_{H'}(v)\cap I(H)\subseteq N_H(v)\cap I(H)$ for all $v\in B(H')$. \label{same_neighbourhood}
\end{enumerate}

\begin{obs}
  Let $H$ and $H'$ be generalized near-triangulations where $H'$ dom-respects $H$ and let $\Delta'$ be a subset of $V(H')$ that dominates $I(H')$ in $H'$.  Then $\Delta'$ dominates $I(H)$ in $H$.
\end{obs}

\begin{proof}
  By \ref{same_inner_vertices}, $I(H)=I(H')$. For each $w\in I(H)=I(H')$, $w\in \Delta'$ or there exists an edge $vw\in E(H')$ with $v\in\Delta'$.  In the latter case, if $v\in B(H')$ then $vw\in E(H)$ by \ref{same_neighbourhood}; if $v\in I(H')$ then $vw\in E(H)$ by \ref{same_inner_vertices}. In either case, $\Delta'$ dominates $w$ in $H$.
\end{proof}

\begin{lem}\label{dom_minimal}
  For any generalized near-triangulation $H$, there exists a dom-minimal generalized near-triangulation $H'$ that dom-respects $H$. 
\end{lem}
\begin{proof}
  The proof is by induction on $|V(H)|+|E(H)|$.  If $H$ is already dom-minimal, setting $H'=H$ satisfies the requirements of the lemma (the base case of the induction); so assume that $H$ is not dom-minimal.  Since \ref{boundary_subset}, \ref{same_inner_vertices}, and \ref{same_neighbourhood} are transitive relations, the dom-respecting relation is transitive: If $H'$ dom-respects $H^*$ and $H^*$ dom-respects $H$, then $H'$ dom-respects $H$.  Therefore, it is sufficient to find $H^*$ with fewer edges or fewer vertices than $H$ that dom-respects $H$, and the inductive hypothesis provides the desired dom-minimal graph $H'$ that dom-respects $H^*$ and $H$.

  If $H$ contains a vertex $v\in B(H)$ with $\deg^+_H(v)=0$, then $H-v$ is a generalized near-triangulation, $B(H-v)\subset B(H)$, $(H-v)[I(H-v)]=H[I(H)]$, and $N_{H-v}(a)\cap I(H)=N_{H}(a)\cap I(H)$ for all $a\in V(H-v)$. Therefore, $H-v$ dom-respects $H$ and has fewer vertices than $H$, so we can apply the inductive hypothesis and be done.  We now assume that $\deg^+_H(v)\ge 1$ for all $v\in B(H)$. If $H$ has a bridge $e=vw$, then $H-e$ is a generalized near-triangulation, $B(H-e)=B(H)$, $(H-e)[I(H-e)]=H[I(H)]$, and $N_{H-e}(a)\cap I(H)=N_{H}(a)\cap I(H)$ for all $a\in V(H-e)$.  Therefore, $H-e$ dom-respects $H$ and has fewer edges than $H$, so we can again apply induction and be done.  We now assume that $H$ is bridgeless. Since $H$ is not dom-minimal, $H$ contains a biconnected component $C$ that is not dom-minimal. (By definition $C$ has at least three vertices.) See \cref{minimal_fig}.
  \begin{figure}
    \centering
    \begin{tabular}{ccc}
      \includegraphics[page=1]{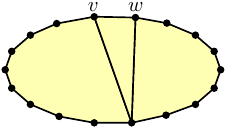} &
      \includegraphics[page=3]{figs/minimal} &
      \includegraphics[page=5]{figs/minimal} \\
      $\Downarrow$ & $\Downarrow$ & $\Downarrow$ \\
      \includegraphics[page=2]{figs/minimal} &
      \includegraphics[page=4]{figs/minimal} &
      \includegraphics[page=6]{figs/minimal}
    \end{tabular}
    \caption{Three cases on the way to making $H$ dom-minimal.}
    \label{minimal_fig}
  \end{figure}
  \begin{itemize}[labelwidth=\widthof{(DM0)}, leftmargin=!]
    \item [\ref{bad_edge}] If there exists an edge $e=vw$ on the outer face of $C$ that is not incident to any inner face $vwx$ with $x\in I(C)$, then $H-e$ is a generalized near-triangulation, $B(H-e)=B(H)$, $(H-e)[I(H-e)]=H[I(H)]$, and $N_{H-e}(a)\cap I(H)=N_{H}(a)\cap I(H)$ for all $a\in V(H-e)$. Therefore, $H-e$ dom-respects $H$ and has fewer edges than $H$.

    \item [\ref{bad_vertex}] If there exists a vertex $v\in B(C)$ with $\deg^+_C(v)=0$ (which can occur if $v$ is a cut vertex of $H$ having inner neighbours in another block), then since $v$ is on the boundary of $C$, $v$ is incident to an edge $e=vw$ on the boundary of the outer face of $C$. Because $\deg^+_C(v)=0$, $v$ has no inner neighbour in $C$, so $e$ is not incident to any inner face $vwx$ with $x\in I(C)$ and we can proceed as in \ref{bad_edge}.  (This eventually leads to all edges of $C$ incident to $v$ being removed from $H$.)

    \item [\ref{inner_degree_1}] If there exists a vertex $v\in B(C)$ with $\deg^+_C(v)=1$ such that $C[N_C[v]]$ is not isomorphic to $K_4$, then $H$ contains faces $xvw$ and $vyw$ where $x,y\in B(C)$ and $w\in I(C)$.  If \ref{bad_edge} does not apply to either of the two edges on the outer face of $C$ incident to $v$, then the edges $vx$ and $vy$ are each on the boundary of the outer face of $C$ and $N_C(v)=\{x,y,w\}$. Since $C[N_C[v]]$ is not isomorphic to $K_4$, $xy\not\in E(H)$.  In this case, let $H^*$ be the graph obtained from $H$ by removing the edge $vw$ and replacing the edges $xv$ and $vy$ with the edge $xy$. Then $H^*$ is a generalized near-triangulation, $B(H^*)=B(H)$, $H^*[I(H^*)]=H[I(H)]$, and $N_{H^*}(a)\cap I(H)\subseteq N_H(a)\cap I(H)$ for every $a\in V(H^*)$, since the only edge $xy$ in $E(H^*)\setminus E(H)$ has $x,y\in B(H)$. Therefore, $H^*$ dom-respects $H$ and has fewer edges than $H$. \qedhere
  \end{itemize}
\end{proof}

\subsection{Finding a \texorpdfstring{$2$--$3$}{2-3} Combo}
\label{combo}

Next, we show that in most cases our algorithm for constructing a connected dominating set is not forced to choose a single vertex of inner degree $2$. Instead, it can choose a pair $v,w$ such that $\deg^+_H(v)=2$ and $\deg^+_{H-v}(w)\ge 3$. We begin by studying the simpler case in which $H$ is a near-triangulation.

\begin{lem}\label{degree_2_outer_neighbour}
  Let $H$ be a dom-minimal near-triangulation. Then either:
  \begin{enumerate}[nosep,nolistsep,label=\rm(\arabic*)]
    \item $H$ is isomorphic to $K_4$;
    \item each vertex $w\in B(H-B(H))$ has a neighbour $v$ in $B(H)$ with $\deg^+_H(v)\ge 2$.
  \end{enumerate}
\end{lem}

\begin{proof}
Suppose that some $w\in B(H-B(H))$ has no neighbour in $B(H)$ of
inner degree at least two. Let $v$ be an arbitrary neighbour of $w$ in $B(H)$.  Thus, $N^+_H(v)=\{w\}$ and $\deg^+_H(v)= 1$. By \ref{inner_degree_1}, $N_H[v]=\{v,w,x,y\}$ induces a $K_4$, where $x,y\in B(H)$. Since $x,y\in B(H)$ and $x,y\in N_H(w)$, we have $N^+_H(x)=N^+_H(y)=\{w\}$, by the choice of $w$.  Thus, $\deg^+_H(x)=\deg^+_H(y)=1$ so, by \ref{inner_degree_1}, $N_H[x]=N_H[y]=\{v,w,x,y\}$ induces a $K_4$ in $H$. Therefore the outer face of $H$ is bounded by $vxy$.  By \cref{useful_little_guy}, the only vertex in the interior of $vxy$ is $w$.  Therefore $V(H)=\{v,w,x,y\}$ and $H$ is isomorphic to $K_4$.
\end{proof}

Note that the next three lemmas (\cref{really_good,leaf_killer,degree_zero_killer}) consider the case where $H$ is a generalized near-triangulation. The following lemma is illustrated in \cref{2_3_figure}.

\begin{figure}
  \centering
  \includegraphics[page=1]{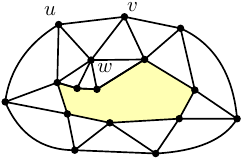}
  \includegraphics[page=2]{figs/2_3_figure}
  \includegraphics[page=3]{figs/2_3_figure}
  \caption{A $2$--$3$ reduction: deleting $v$, which has inner degree $2$, exposes $w$, whose inner degree is at least $3$. After deleting $w$, the vertex $u$ has inner degree zero and is eliminated when making $H-\{v,w\}$ dom-minimal.}
  \label{2_3_figure}
\end{figure}

\begin{lem}\label{really_good}
  Let $H$ be a dom-minimal generalized near-triangulation.  Then either:
  \begin{enumerate}[label=\rm(\arabic*),nosep,nolistsep] 
    \item $H-B(H)$ is critical; \label{two_critical}
    \item $B(H)$ contains a vertex $v$ with $\deg^+_H(v)\ge 3$; or \label{degree_three}
    \item $H$ contains distinct vertices $v_0$, $v_j$, and $w$ such that
    \begin{enumerate}[label=\rm(\alph*),nosep,nolistsep]
      \item\label{deg_v_0} $v_0\in B(H)$ and $\deg^+_H(v_0)=2$;
      \item\label{deg_w} $w\in I(H)\cap B(H-v_0)$, $v_0w\in E(H)$, and $\deg^+_{H-v_0}(w)\ge 3$; and
      \item\label{deg_v_j} $v_j\in B(H)$, $N^+_H(v_j)\subseteq N_H[w]$, and $\deg^+_{H-\{v_0,w\}}(v_j)=0$.
    \end{enumerate}
    \label{two_three_pair}
  \end{enumerate}
\end{lem}

\begin{proof}
  We will assume that $H$ does not satisfy \ref{two_critical} or \ref{degree_three} and show that $H$ must satisfy \ref{two_three_pair}.  Since $H-B(H)$ is not critical, $B(H-B(H))$ contains a vertex $w$ with $\deg^+_{H-B(H)}(w)\ge 2$.

  Let $C$ be the biconnected component of $H$ that contains $w$.  Then $C$ is a near-triangulation and we can apply \cref{degree_2_outer_neighbour} to $C$ and $w$. The first alternative in \cref{degree_2_outer_neighbour} is incompatible with the assumption that $\deg^+_{H-B(H)}(w)\ge 2$.  Therefore, we conclude that $N_H(w)\cap B(H)$ contains a vertex $v_0$ with $\deg^+_H(v_0)\ge 2$.  Since $H$ does not satisfy \ref{degree_three}, $\deg^+_H(v_0)< 3$, so $\deg^+_H(v_0)=2$. Thus, the vertex $v_0$ satisfies \ref{deg_v_0}. Refer to \cref{really_good_fig}.
  \begin{figure}
    \centering
    \begin{tabular}{ccc}
      \includegraphics[page=1]{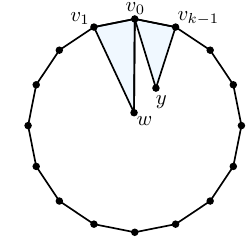} &
      \includegraphics[page=2]{figs/really_good} &
      \includegraphics[page=3]{figs/really_good}
    \end{tabular}
    \caption{The proof of \cref{really_good}.
    }
    \label{really_good_fig}
  \end{figure}
  Let $v_0,\ldots,v_{k-1}$ be the cycle that bounds the inner face $f$ of $H[B(H)]$ that contains $w$ in its interior. In the remainder of this proof, all subscripts on $v$ are implicitly modulo $k$.  Since $H$ is a generalized near-triangulation, $H$ contains triangles $v_0v_1x$ and $v_{k-1}v_0 y$ with $x$ and $y$ in the interior of or on the boundary of $f$.  Since $f$ is a face of $H[B(H)]$, each of $x$ and $y$ is in the interior of $f$.  If $x=y$, then $x$ is the only neighbour of $v_0$ in the interior of $f$, and hence $w=x$. If $x\ne y$, at least one of $x$ or $y$ is equal to $w$, since otherwise $\deg^+_H(v_0)\ge 3$. Assume, without loss of generality, that $x=w$.  (If $y=w$, then we can reverse the cyclic order $v_0,\ldots,v_{k-1}$ and swap the roles of $x$ and $y$.)  Therefore $v_0v_1 w$ is an inner face of $H$.

  Let $j\ge 1$ be the maximum integer such that $v_{a-1}v_{a}w$ is an inner face of $H$ for all $a\in\{1,\ldots,j\}$.  Note that $j\le k-1$ since, otherwise, the component of $H-B(H)$ that contains $w$ contains only a single vertex, contradicting the fact that $\deg^+_{H-B(H)}(w)\ge 2$.

  Since $H$ is a generalized near-triangulation and $f$ is a face of $H[B(H)]$, $H$ has some face $v_j v_{j+1} z$ with $z$ in the interior of $f$.  By the definition of $j$, $z\neq w$.  Therefore, $N_H^+(v_j)\supseteq \{w,z\}$ and, since $\deg^+_H(v_j)\le 2$, $N_H^+(v_j)= \{w,z\}$.  Since $H$ is a generalized near-triangulation with $w,z\in I(H)$ that contains the faces $v_{j-1}v_j w$ and $v_{j}v_{j+1}z$, and $N_H^+(v_j)=\{w,z\}$, it follows that $v_jzw$ is a face of $H$.  In particular, $wz$ is an edge of $H$. Therefore, $N_H[w]\supseteq \{w,z\}=N^+_H(v_j)$, so $v_j$ satisfies the first requirement in \ref{deg_v_j}.  Since $v_0\in B(H)$ and $v_0w\in E(H)$, $w\in B(H-v_0)$.  Since $wz\in E(H)$, $z\in B(H-\{v_0,w\})$.  Therefore $N^+_{H-\{v_0,w\}}(v_j)\subseteq N^+_H(v_j)\setminus\{w,z\}=\emptyset$.  Thus, $\deg^+_{H-\{v_0,w\}}(v_j)=0$, so $v_j$ satisfies the second requirement in \ref{deg_v_j}.
  
  All that remains is to show that $w$ satisfies \ref{deg_w}.  First, observe that $z$ is in $B(H-B(H))$, so $z$ does not contribute to $\deg^+_{H-B(H)}(w)$. We claim that $z\in I(H-v_0)$, so $z$ does contribute to $\deg^+_{H-v_0}(w)$. Indeed, the only other possibility is that $z$ is adjacent to $v_0$.  In this case, consider the cycle $D:=v_0,\ldots,v_j,z$.  Then $D$ has $w$ in its interior and the vertices of $N^+_{H-B(H)}(w)$ are also in the interior of $D$. For each $a\in\{1,\ldots,j\}$, $v_{a-1}v_a w$ is a face of $H$ and $v_jwz$ is a face of $H$, so the cycle $D_1:=v_0,\ldots,v_j,z,w$ does not contain any vertices of $N^+_{H-B(H)}(w)$ in its interior.  Therefore, the vertices in $N^+_{H-B(H)}(w)$ must be in the interior of the cycle $D_2:=v_0,w,z$.  By \cref{useful_little_guy}, $v_0$ is adjacent to some vertex $u$ in the interior of $D_2$.  Thus, $N^+_H(v_0)\supseteq\{w,z,u\}$, which is a contradiction since $\deg^+_H(v_0)=2$.  Therefore $v_0$ is not adjacent to $z$, so $z\in I(H-v_0)$.  Therefore $N^+_{H-v_0}(w)\supseteq N^+_{H-B(H)}(w)\cup\{z\}$. Therefore $\deg^+_{H-v_0}(w)\ge 3$, so $w$ satisfies \ref{deg_w}.
\end{proof}

We are interested in \cref{really_good} because of the impact it has on the number of internal vertices in $H-\{v,w\}$ and the number of boundary vertices in a dom-minimal graph $H'$ that dom-respects $H-\{v,w\}$. This is summarized by the following corollary.

\begin{cor}\label{two_three_killer}
  Let $H$ be a dom-minimal generalized near-triangulation such that $\deg^+_H(v)\le 2$ for all $v\in B(H)$, and such that $H-B(H)$ is not critical. Then there exist vertices $v\in B(H)$ and $w\in N^+_H(v)$ such that $t:=\deg^+_{H-v}(w)\ge 3$ and any dom-minimal generalized near-triangulation $H'$ that dom-respects $H-\{v,w\}$ satisfies:
\[
  |I(H')|=|I(H)|-(t+2)
  \qquad\text{and}\qquad
  |B(H')|\le |B(H)|+t-1 \enspace .
\]
\end{cor}

\begin{proof}
  Since $H-B(H)$ is not critical and $\deg^+_H(v)\le 2$ for all $v\in B(H)$, $H$ satisfies neither \ref{two_critical} nor \ref{degree_three} of \cref{really_good}. Therefore, \cref{really_good} guarantees the existence of distinct vertices $v,u\in B(H)$ and $w\in I(H)$ such that $w\in N^+_H(v)$, $\deg^+_H(v)=2$, $t:=\deg^+_{H-v}(w)\ge 3$, and $\deg^+_{H-\{v,w\}}(u)=0$.

  Let $\Delta:=\{v,w\}$. Then $I(H-\Delta) = I(H)\setminus (N^+_H(v)\cup N^+_{H-v}(w))$. Since $w\in N^+_H(v)$, these two sets are disjoint, and their sizes are $\deg^+_H(v)=2$ and $\deg^+_{H-v}(w)=t$. By \ref{same_inner_vertices}, $I(H')=I(H-\Delta)$, so $|I(H')|=|I(H)|-(t+2)$.

  Now consider $B(H')$. Although $u\in B(H-\Delta)$, $\deg^+_{H-\Delta}(u)=0$, so property \ref{bad_vertex} of dom-minimality implies that $u\not\in V(H')$. Furthermore, any vertex in $B(H-\Delta)$ is either in $B(H)\setminus\{v\}$ or an inner vertex of $H$ adjacent to $\Delta$, which belongs to $(N^+_H(v)\setminus\{w\})\cup N^+_{H-v}(w)$. Therefore,
  \[
    B(H')\subseteq (B(H)\setminus\{v,u\})\cup (N^+_H(v)\setminus\{w\})\cup N^+_{H-v}(w) \enspace ,
  \]
  so
  \[
    |B(H')|\le |B(H)|-2+1+t = |B(H)|+t-1 \enspace . \qedhere
  \]
\end{proof}

\subsection{Eliminating Inner Leaves}
\label{one_kill_sec}

Next we show that, even when all vertices in $B(H)$ have inner degree at most $2$, we can still make an efficient move if $H-B(H)$ has a vertex of degree $1$.  In particular, we can choose $v\in B(H)$ whose removal reduces the number of internal vertices by $2$ while decreasing the number of boundary vertices. 

\begin{lem}\label{leaf_killer}
  Let $H$ be a generalized near-triangulation such that $\deg^+_H(v)\le 2$ for all $v\in B(H)$, and suppose that $H-B(H)$ has a vertex of degree one. Then there is a vertex $v\in B(H)$ such that any dom-minimal generalized near-triangulation $H'$ that dom-respects $H-v$ satisfies:
\[
 |I(H')|=|I(H)|-2,
 \qquad |B(H')|\le |B(H)|-1.
\]
\end{lem}

\begin{figure}[htbp]
    \centering
    \begin{tabular}{ccc}
      \includegraphics[page=1]{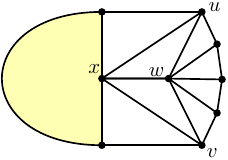} &
      \includegraphics[page=2]{figs/killing_a_leaf} &
      \includegraphics[page=3]{figs/killing_a_leaf} \\
      $H$ & $H-v$ & $H'$
    \end{tabular}
    \caption{Eliminating an inner leaf; additional boundary vertices of inner degree zero may also be discarded.}
    \label{killing_a_leaf}
  \end{figure}

\begin{proof}[Proof of \cref{leaf_killer}]
Let $w$ be a vertex of degree one in $H-B(H)$, and let $\{x\}:=N_{H-B(H)}(w)$, so $x$ is the unique
neighbour of $w$ in $H-B(H)$. Then $w\in V(H-B(H))=I(H)$, so $\deg_H(w)\ge 3$.  Since $\deg_{H-B(H)}(w)=1$, $w\in B(H-B(H))$.  Since $\deg_{H-B(H)}(w)=1$, the edge $wx$ is on the outer face of $H-B(H)$, so $x\in B(H-B(H))$.  Therefore $N^+_{H-B(H)}(w)\subseteq N_{H-B(H)}(w)\setminus \{x\}=\emptyset$.

Since $w\in I(H)$, $H$ contains two inner faces $wxv$
and $wxu$, with $v\ne u$. Then $\{v,u\}\subseteq B(H)$, since $v,u\in N_H(w)$ but $v,u\not\in N_{H-B(H)}(w)=\{x\}$. Thus $\{w,x\}\subseteq N_H^+(v)\cap N_H^+(u)$. Since $\deg^+_H(v),\deg^+_H(u)\le 2$, this implies that $N_H^+(v)=N_H^+(u)=\{w,x\}$.
Therefore, $B(H-v)=(B(H)\setminus\{v\})\cup\{w,x\}$ and $\deg^+_{H-v}(u)=0$.

Let $H'$ be a dom-minimal generalized near-triangulation that dom-respects $H-v$.  By \ref{same_inner_vertices}, $I(H')=I(H-v)=I(H)\setminus\{x,w\}$ so $|I(H')|=|I(H)|-2$.  Since $w\in B(H-v)$ and $\deg^+_{H-v}(w)=0$, $w\not\in V(H')$, by \ref{bad_vertex}.  Since $u\in B(H-v)$ and $\deg^+_{H-v}(u)=0$, $u\not\in V(H')$, by \ref{bad_vertex}.  Therefore $B(H')\subseteq B(H-v)\setminus\{u,w\}=(B(H)\cup\{x\})\setminus\{u,v\}$, so $|B(H')|\le |B(H)|+1-2 = |B(H)|-1$.
\end{proof}

\subsection{Eliminating Inner Isolated Vertices}
\label{zero_kill_sec}

We now show that, even when all vertices in $B(H)$ have inner degree at most $2$ and $H-B(H)$ has no degree-$1$ vertices, we can still make an efficient choice if $H-B(H)$ has a vertex of degree $0$.

\begin{lem}\label{degree_zero_killer}
  Let $H$ be a dom-minimal generalized near-triangulation such that $\deg^+_H(v)\le 2$ for all $v\in B(H)$, and suppose that $H-B(H)$ has an isolated vertex but no vertex of degree one. Then there is a vertex $v\in B(H)$ such that any dom-minimal graph $H'$ that dom-respects $H-v$ satisfies one of the following:
\begin{enumerate}[label=\rm(\alph*),nosep,nolistsep]
\item\label{decrease_inner_by_two} $|I(H')|=|I(H)|-2$ and $|B(H')|\le|B(H)|-1$;
\item\label{decrease_boundary_by_three} $|I(H')|=|I(H)|-1$ and $|B(H')|\le|B(H)|-3$.
\end{enumerate}
\end{lem}

\begin{figure}%
    \centering
    \begin{tabular}{ccc}
      $H$ & $H-v$ & $H'$ \\[2ex]
      \includegraphics[page=1]{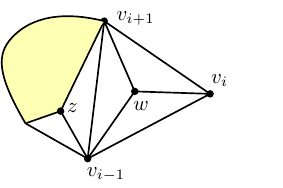} &
      \includegraphics[page=2]{figs/dzk} &
      \includegraphics[page=3]{figs/dzk} \\[2ex]
      \includegraphics[page=4]{figs/dzk} &
      \includegraphics[page=5]{figs/dzk} &
      \includegraphics[page=6]{figs/dzk} \\[2ex] \hline \\[2ex]
      \includegraphics[page=7]{figs/dzk} &
      \includegraphics[page=8]{figs/dzk} &
      \includegraphics[page=9]{figs/dzk} \\[5ex]
      \includegraphics[page=13]{figs/dzk} &
      \includegraphics[page=14]{figs/dzk} &
      \includegraphics[page=15]{figs/dzk} \\[5ex]
      \includegraphics[page=10]{figs/dzk} &
      \includegraphics[page=11]{figs/dzk} &
      \includegraphics[page=12]{figs/dzk} 
    \end{tabular}
    
    \caption{Eliminating an isolated vertex $w$ in $H-B(H)$.}
    \label{isolated_fig}
  \end{figure}

\begin{proof}
Let $w$ be an isolated vertex of $H-B(H)$. Let $f$ be the (inner) face of $H[B(H)]$ that has $w$ in its interior. Let $v_0,v_1,\dots,v_{k-1}$ be the vertices of $f$ in clockwise order.  For the rest of this proof, let $v_{i}=v_{i\bmod k}$ for every integer $i$.  Thus, $N_H(w)=V(f)$ and $wv_iv_{i+1}$ is an inner face of $H$ for each $i\in\{0,\ldots,k-1\}$.

\begin{itemize}[leftmargin=*,label=$\bigstar$]
\item Suppose first that $\deg^+_H(v_i)=1$ for some $i\in\{0,\ldots,k-1\}$. (See the top part of \cref{isolated_fig}.) By \ref{inner_degree_1}, $H[N_H[v_i]]=\{v_{i-1},v_i,v_{i+1},w\}$ induces a $K_4$ in $H$.  
\begin{itemize}[leftmargin=*,label=\ding{58}]

\item If $\deg^+_{H}(v_{i-1})=2$, then $N^+_H(v_{i-1})=\{w,z\}$ for some $z\in I(H)$ (see the first row of \cref{isolated_fig}). Let $v:=v_{i-1}$. Then $I(H-v)=I(H)\setminus \{w,z\}$ and $B(H-v)=(B(H)\setminus \{v\})\cup \{w,z\}$.  Applying \cref{dom_minimal} to $H-v$ gives a dom-minimal graph $H'$ that dom-respects $H-v$.  Since $w\in B(H-v)$ and $\deg^+_{H-v}(w)=0$, \ref{bad_vertex} implies that $w\not\in V(H')$, so $w\not\in B(H')$.  Since $N^+_H(v_i)=\{w\}$ and $w\in B(H-v)$, $\deg^+_{H-v}(v_i)=0$, so \ref{bad_vertex} implies that $v_i\not\in V(H')$, so $v_i\not\in B(H')$.  Therefore, $|B(H')|\le |B(H-v)|-2= |B(H)|-1$.  By \ref{same_inner_vertices}, $|I(H')|=|I(H-v)|=|I(H)|-2$.  Thus $H'$ satisfies \ref{decrease_inner_by_two}.  Similarly, if $\deg^+_H(v_{i+1})=2$, then choosing $v:=v_{i+1}$ establishes the lemma.  

\item Therefore, we may assume that $\deg^+_H(v_{i-1})=\deg^+_H(v_{i+1})=1$. (See the second row of \cref{isolated_fig}.) Let $v:=v_i$. Then $I(H-v)=I(H)\setminus \{w\}$ and $B(H-v)=(B(H)\setminus \{v\})\cup\{w\}$.  Let $H'$ be obtained by applying \cref{dom_minimal} to $H-v$.  Since $w\in B(H-v)$ and $\deg^+_{H-v}(w)=0$, \ref{bad_vertex} implies that $w\not\in V(H')$, so $w\not\in B(H')$.  Since $N^+_H(v_{i-1})=N^+_H(v_{i+1})=\{w\}$ and $w\in B(H-v)$, $\deg^+_{H-v}(v_{i-1})=\deg^+_{H-v}(v_{i+1})=0$, so \ref{bad_vertex} implies that $v_{i-1},v_{i+1}\not\in V(H')$, so $v_{i-1},v_{i+1}\not\in B(H')$.  Therefore, $|B(H')|\le |B(H-v)|-3= |B(H)|-3$.  By \ref{same_inner_vertices}, $|I(H')|=|I(H-v)|=|I(H)|-1$.  Thus $H'$ satisfies \ref{decrease_boundary_by_three}.
\end{itemize}
\item What remains is to consider the possibility that $\deg^+_H(v_i)=2$ for each $i\in\{0,\ldots,k-1\}$. (See the bottom part of \cref{isolated_fig}.)  By the definition of $f$, $N^+_H(v_i)$ contains $w$ for each $i\in\{0,\ldots,k-1\}$.  For each $i\in\{0,\ldots,k-1\}$, let $\{w,z_i\}:=N^+_H(v_i)$.  
\begin{itemize}[leftmargin=*,label=\ding{58}]
\item If $z_i=z_{i+1}$ for some $i\in\{0,\ldots,k-1\}$, then we choose $v:=v_i$. (See the third row of \cref{isolated_fig}.)  In this case, $I(H-v)=I(H)\setminus\{w,z_i\}$ and $B(H-v)=(B(H)\setminus \{v\})\cup \{w,z_i\}$. Since $\deg^+_{H-v}(w)=0$ and $\deg^+_{H-v}(v_{i+1})=0$, applying \cref{dom_minimal} to $H-v$ yields a graph $H'$ that does not contain $w$ or $v_{i+1}$ by \ref{bad_vertex}.  Thus, $|I(H')|=|I(H)|-2$ and $|B(H')|\le |B(H-v)|-2 = |B(H)|-1$, so $H'$ satisfies \ref{decrease_inner_by_two}.  
\item All that remains is to consider the case in which $z_i\neq z_{i+1}$ for each $i\in\{0,\ldots,k-1\}$.  
\begin{itemize}[leftmargin=*,label=\ding{70}]
\item If there exists $i\in\{0,\ldots,k-1\}$ such that $\deg_{H-B(H)}(z_i)=0$, then we choose $v:=v_i$. (See the fourth row in \cref{isolated_fig}.)  In this case, $I(H-v)=I(H)\setminus\{w,z_i\}$ and $B(H-v)=(B(H)\setminus\{v\})\cup\{w,z_i\}$.  Since $\deg^+_{H-v}(w)=0$ and $\deg^+_{H-v}(z_i)=0$, applying \cref{dom_minimal} to $H-v$ yields a graph $H'$ that does not contain $w$ or $z_i$ by \ref{bad_vertex}.  Thus $|I(H')|=|I(H)|-2$ and $|B(H')|\le |B(H-v)|-2 = |B(H)|-1$, so $H'$ satisfies \ref{decrease_inner_by_two}.  
\item We now assume that $\deg_{H-B(H)}(z_i)\neq 0$ for each $i\in\{0,\ldots,k-1\}$. (See the fifth row of \cref{isolated_fig}.) Since the statement of the lemma assumes that $\deg_{H-B(H)}(z)\neq 1$ for each $z\in I(H)$, this implies that $\deg_{H-B(H)}(z_i)\ge 2$ for each $i\in\{0,\ldots,k-1\}$.

We claim that $z_{0}$ and $z_{1}$ are in different connected components of $H-B(H)$.  Suppose, for the sake of contradiction, that this is not the case.  Then $z_0$ and $z_1$ are in the interior of the same inner face $g$ of the outerplane graph $H[B(H)]$. Then $g\neq f$ since $w$ is the only vertex of $I(H)$ in the interior of $f$.  The face $g$ has $v_0$ and $v_1$ on its boundary.  Since $H[B(H)]$ is outerplane, this implies that $g$ has the edge $v_0v_1$ on its boundary.  Since $H$ is a generalized near-triangulation, $H$ contains a face $v_0v_1p$ with $p$ in the closure of $g$. Since $g\neq f$ and $w$ is in the interior of $f$, $p\neq w$.  If $p\in B(H)$, then $p$ is on the boundary of $g$, so $g=v_0v_1p$, which is not possible since the interior of $g$ contains vertices $z_0$ and $z_1$.  If $p\in I(H)$, then $p\in N^+_H(v_0)\cap N^+_H(v_1)$, which is also not possible since $N^+_H(v_0)\cap N^+_H(v_1)=\{w\}\neq\{p\}$.  Thus, $z_0$ and $z_1$ are in different components of $H-B(H)$ as claimed.

We deal with this final case by choosing $v:=v_0$. Then $I(H-v)=I(H)\setminus\{w,z_0\}$ and $B(H-v)=(B(H)\setminus\{v\})\cup\{w,z_0\}$. Applying \cref{dom_minimal} to $H-v$ yields a dom-minimal graph $H'$ that dom-respects $H-v$. Since $w\in B(H-v)$ and $\deg^+_{H-v}(w)=0$, \ref{bad_vertex} implies that $w\not\in V(H')$. We claim that $v_1\not\in V(H')$.

Suppose, for the sake of contradiction, that $v_1\in V(H')$. By \ref{boundary_subset}, $v_1\in B(H')$. By \ref{bad_vertex}, $\deg^+_{H'}(v_1)\ge 1$. Since $z_1\not\in N_H(v_0)$, $z_1\in I(H-v)$, and by \ref{same_neighbourhood}, $N^+_{H'}(v_1)\subseteq N^+_{H-v}(v_1)=\{z_1\}$, so $N^+_{H'}(v_1)=\{z_1\}$. Since $\deg^+_{H'}(v_1)=1$, \ref{inner_degree_1} implies that $H'[N_{H'}[v_1]]\cong K_4$. Let $\{v_1,z_1,x,y\}=N_{H'}[v_1]$, where $x,y\in B(H')$. Since $z_1\in I(H')$, \ref{same_neighbourhood} implies that $xz_1,yz_1\in E(H)$. Then $x,y\neq w$ (since $w\not\in V(H')$) and $x,y\neq z_0$ (since $z_0$ and $z_1$ are in different components of $H-B(H)$). Thus $x,y\in B(H)$, so $\deg^+_H(x)\le 2$ and $\deg^+_H(y)\le 2$.

In $H'$, the cycle $v_1z_1x$ has no vertices of $H'$ in its interior since, otherwise, \cref{useful_little_guy} part~\ref{useful_little_guy:1} would imply that $N^+_{H'}(v_1)$ contains a vertex in the interior of $v_1z_1x$, contradicting $N^+_{H'}(v_1)=\{z_1\}$. For the same reason, the cycle $v_1yz_1$ has no vertices of $H'$ in its interior. Recall that $\deg_{H-B(H)}(z_1)\ge 2$. By \ref{same_inner_vertices}, $H'[I(H')]=(H-v)[I(H-v)]$, so $N_{H-B(H)}(z_1)\subseteq I(H')$. Since $v_1z_1x$ and $v_1yz_1$ have no vertices of $H'$ in their interiors, all vertices of $N_{H-B(H)}(z_1)$ must lie in the interior of the cycle $z_1yx$ in $H'$. By \cref{useful_little_guy} part~\ref{useful_little_guy:2} applied to $H'$, at least one of $x$ or $y$ has at least two neighbours in the interior of $z_1yx$ in $H'$. Without loss of generality, suppose that $x$ has two neighbours $p,q$ in the interior of $z_1yx$. Since $p,q\in I(H')$, \ref{same_neighbourhood} implies that $xp,xq\in E(H)$. Together with $xz_1\in E(H)$, we obtain $N^+_H(x)\supseteq\{p,q,z_1\}$, so $\deg^+_H(x)\ge 3$, a contradiction.

This contradiction proves that $v_1\not\in V(H')$. Since neither $w$ nor $v_1$ belongs to $V(H')$, $|B(H')|\le |B(H-v)|-2 = |B(H)|-1$. By \ref{same_inner_vertices}, $|I(H')|=|I(H-v)|=|I(H)|-2$. Thus $H'$ satisfies \ref{decrease_inner_by_two}. \qedhere 
\end{itemize}
\end{itemize}
\end{itemize}
\end{proof}

\begin{figure}[htbp]
  \centering
  \includegraphics[page=1, scale=0.7]{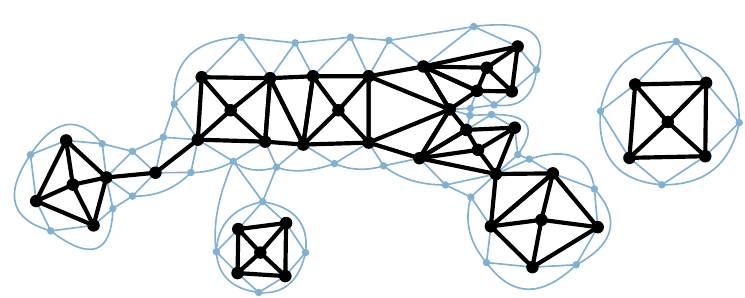}
  \caption{A $2$-critical generalized near-triangulation.}
  \label{two_critical_figure}
\end{figure}

\subsection{\texorpdfstring{$2$}{2}-Critical Graphs}
\label{two_critical_section}

After trying all the reductions introduced so far, we are left with a graph $H$ that does not have a vertex $v$ with $\deg^+_H(v)\ge 3$, does not have a $2$-$3$ pair from \cref{really_good}, and every vertex of $H-B(H)$ has degree at least $2$ (by \cref{leaf_killer,degree_zero_killer}). This leads to the following
definition:
A generalized near-triangulation $H$ is \defin{$2$-critical} if
\begin{enumerate}[label=(2-C\arabic*),labelwidth=\widthof{(2-C2)},leftmargin=!]
  \item $\deg^+_H(v)\le 2$ for each $v\in B(H)$; \label{inner_degree_2}
  \item $H-B(H)$ is critical; \label{inner_critical} and
  \item $\deg_{H-B(H)}(w)\ge 2$ for all $w\in V(H-B(H))$. \label{inner_vertex_degree_2}
\end{enumerate}
(See \cref{two_critical_figure}.) We will work our way up to a proof of the following lemma, which allows our algorithm to handle $2$-critical graphs directly, in one step:

\begin{lem}\label{two_critical_handler}
  Let $H$ be a $2$-critical generalized near-triangulation.  Then there exists $X\subseteq V(H)$ of size at most $(2|B(H-B(H))| + |I(H-B(H))|)/3$ that dominates $I(H)$ and such that each component of $H[X]$ contains at least one vertex in $B(H)$.
\end{lem}

\begin{lem}\label{two_critical_boundary_size}
  Let $H$ be a $2$-critical generalized near-triangulation. Then $|B(H)|\ge |B(H-B(H))|$.
\end{lem}

\begin{proof}
  If $H$ is the empty graph, then the result is trivial. If $H$ is non-empty but $H-B(H)$ is the empty graph, then the result is also trivial. Now assume that $H-B(H)$ is non-empty.

  Let $\mathcal{C}$ be the set of components of $H-B(H)$.
  Let $C$ be a component in $\mathcal{C}$ and define the set $B_2(C)\subseteq B(H)$ as follows: Let $w_0,\ldots,w_k$ be the clockwise walk around the outer face of $C$, so that $w_0=w_k$ and $|B(C)|=|\{w_1,\ldots,w_k\}|\le k$.  If an edge $vw$ of $H-B(H)$ appears more than once in $w_0,\ldots,w_k$, then $vw$ is necessarily a bridge in $C$ and is traversed once from $v$ to $w$ and a second time from $w$ to $v$. That is, $(v,w)=(w_{i-1},w_{i})$ and $(w,v)=(w_{j-1},w_{j})$ for distinct $i,j\in\{1,\ldots,k\}$. For each $i\in\{1,\ldots,k\}$, $H$ contains an inner face $w_{i-1}w_iv_i$ that is to the left of the edge $w_{i-1}w_i$ when traversed from $w_{i-1}$ to $w_i$.  Note that $v_i\in B(H)$ for each $i\in\{1,\ldots,k\}$. The vertices $v_1,\ldots,v_k$ are distinct: For distinct $i,j\in\{1,\ldots,k\}$, the same vertex cannot be adjacent to both $w_{i-1}w_{i}$ and to $w_{j-1}w_{j}$ since this would imply that $N^+_H(v_i)\supseteq\{w_{i-1},w_{i},w_{j-1},w_{j}\}$.  If $|\{w_{i-1},w_{i},w_{j-1},w_{j}\}|>2$, then this contradicts the fact that $\deg^+_H(v_i)\le 2$ (by \ref{inner_degree_2}).  If $|\{w_{i-1},w_{i},w_{j-1},w_{j}\}|=2$, then $w_{i-1}w_i$ is a bridge of $C$ and $w_{i-1}=w_{j}$ and $w_{i}=w_{j-1}$.  This implies that $H$ contains two bounded faces, one bounded by a cycle through $v_i,w_{i-1},w_i$ and another bounded by a cycle through $v_i,w_i,w_{i-1}$.  Since $H$ is a simple graph, this is not possible.  Define $B_2(C):=\{v_1,\ldots,v_k\}$.  Then $|B_2(C)|= k\ge |B(C)|$. 
  
  For any two distinct components $C$ and $C'$ of $H-B(H)$, $B_2(C)$ and $B_2(C')$ are disjoint since any vertex $v\in B_2(C)\cap B_2(C')$ would have $\deg^+_H(v)\ge|N_H(v)\cap V(C)|+|N_H(v)\cap V(C')|\ge 2+2 = 4 > 2$, contradicting \ref{inner_degree_2}.  Therefore, $|B(H)|\ge \sum_{C\in\mathcal{C}} |B_2(C)|\ge \sum_{C\in\mathcal{C}} |B(C)|=|B(H-B(H))|$.
\end{proof}

For each integer $r\ge 3$, the \defin{$r$-wheel} $W_r$ consists of an outer cycle $v_0,\ldots,v_{r-1}$ and a single inner vertex $x$ adjacent to every vertex of the cycle. For even $r$, $W_r$ is an \defin{even wheel}.

To establish \cref{two_critical_handler}, we will show in \cref{three_sets_coverage} that $V(H)$ contains three dominating sets $X_0, X_1, X_2$ whose total size is at most $2|B(H-B(H))| + |I(H-B(H))|$; taking the smallest of the three then immediately yields \cref{two_critical_handler}. We build up to this via a block-cut decomposition of each component of $H-B(H)$:
\begin{itemize}[nosep,nolistsep]
  \item \cref{biconnected_critical} partitions biconnected critical blocks (except even wheels) using a proper $3$-colouring of an outerplanar graph;
  \item \cref{wheelie} handles even wheels, which require a specialized $3$-partition;
  \item \cref{critical} pastes these partitions along cut vertices across the block-cut tree of each connected component; and
  \item \cref{three_sets_coverage} anchors each part to a vertex in $B(H)$, producing the desired three dominating sets $X_0,X_1,X_2$.
\end{itemize}
Note that the next three lemmas (\cref{biconnected_critical,wheelie,critical}) concern \emph{critical} graphs rather than $2$-critical graphs, and are illustrated in \cref{biconnected_critical_colouring,even_wheel,critical_colouring}.

\begin{figure}
  \centering
  \includegraphics[page=7,trim={0 55 0 10},clip]{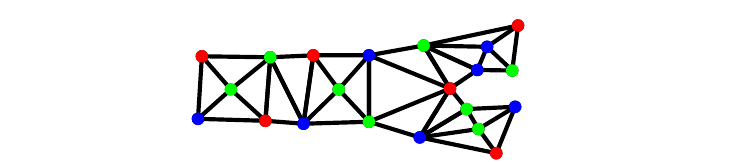}
  \caption{Partitioning the vertices of a biconnected critical graph into three dominating sets in the proof of \cref{biconnected_critical}.}
  \label{biconnected_critical_colouring}
\end{figure}

\begin{lem}\label{biconnected_critical}
  Let $H$ be a biconnected critical generalized near-triangulation with at least $3$ vertices and not isomorphic to $W_k$ for any even integer $k$.  Then there exists a partition $\{X_0,X_1,X_2\}$ of $V(H)$ such that
  \begin{enumerate}[label=\rm(\roman*),nosep,nolistsep]
    \item for each edge $vw$ of $H[B(H)]$, $v\in X_i$ and $w\in X_j$ for some $i\neq j$;\label{proper}
    \item for each $i\in\{0,1,2\}$, $X_i$ dominates $H$; and \label{dominates_h}
    \item each $x\in I(H)\cap X_i$ has a neighbour in $B(H)\cap X_i$. \label{inner_attachment}
  \end{enumerate}
\end{lem}

\begin{figure}
    \centering
    \begin{tabular}{ccc}
      \includegraphics[page=1]{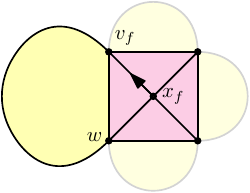} &
      \includegraphics[page=2]{figs/biconnected} &
      \includegraphics[page=3]{figs/biconnected} \\
    \end{tabular}
    \caption{The local construction in \cref{biconnected_critical}.}
    \label{contraction_proof}
  \end{figure}

\begin{proof}
Let $Q:=H[B(H)]$. Since $H$ is biconnected, the outer face of $H$ (and of $Q$) is bounded by a cycle through $B(H)$.  Recall that an inner face of $Q$ is \emph{marked} if it contains a vertex of $I(H)$ in its interior.  By \cref{critical_structure}, each marked face of $Q$ contains exactly one vertex of $H$ in its interior. Each unmarked bounded face of $Q$ is also a bounded face of $H$, and is therefore a triangle.

First we deal directly with the case in which $Q$ has only one bounded face $f_0$. If $f_0$ is unmarked, then $Q=H$ has three vertices $u$, $v$, and $w$.  Setting $X_1:=\{u\}$, $X_2:=\{v\}$, and $X_0:=\{w\}$ satisfies all requirements.  If $f_0$ is marked, then $H$ is a wheel graph $W_k$ with $k$ odd. Let $v_0,\ldots,v_{k-1}$ be the cycle that bounds the outer face of $H$ and let $x$ be the single vertex in $I(H)$.  In this case, we choose $X_0:=\{x,v_0\}$, $X_1:=\{v_1,v_3,\ldots,v_{k-2}\}$ and $X_2:=\{v_2,v_4,\ldots,v_{k-1}\}$.  This also satisfies our requirements, for the following reasons: Each of $X_0\cap B(H)$, $X_1$ and $X_2$ is an independent set in $Q$, which satisfies \ref{proper}. For each $b\in\{1,\ldots,k-1\}$, the vertex $v_b$ is in $X_j$ for some $j\in\{1,2\}$, and $v_b$ is adjacent to $x\in X_0$. If $b\ge 2$, then $v_{b-1}\in X_{3-j}$. Otherwise, $v_{b+1}\in X_{3-j}$. Since $\{0,j,3-j\}=\{0,1,2\}$, $N_H[v_b]$ contains a vertex in each of $X_0$, $X_{j}$, and $X_{3-j}$. The vertex $v_0$ is in $X_0$ and is adjacent to $v_1\in X_1$ and to $v_{k-1}\in X_2$.  Thus, $N_H[v]$ contains a vertex in each of $X_0$, $X_1$, and $X_2$ for all $v\in V(H)$, so $X_0,X_1,X_2$ satisfies \ref{dominates_h}.  The single vertex $x\in I(H)$ is in $X_0$ and has a neighbour $v_0\in B(H)$ that is also in $X_0$, which satisfies \ref{inner_attachment}.  

We now assume that $Q$ has at least two bounded faces. Thus, each bounded face $f$ of $Q$ has a \emph{chord} of $Q$---an edge on the boundary of two bounded faces of $Q$ (one of which is $f$).  For each inner face $f$ of $Q$, let $v_f$ be an endpoint of a chord of $f$.  For each marked inner face $f$ of $Q$, let $x_f$ be the unique vertex in $I(H)$ that is contained in the interior of $f$.

Triangulate $Q$ as follows:  For each marked face $f$ of $Q$, add edges in the interior of $f$ incident to $v_f$ so that $v_f$ is adjacent to all vertices of $f$. Since every inner face of $H$ is a triangle, each unmarked face of $Q$ is a triangle. Therefore, this procedure results in a graph $M\supseteq Q$ where every inner face is a triangle and all vertices of $M$ appear on the outer face of $M$.  Thus, $M$ is a maximal outerplanar graph.  Properly $3$-colour $M$ using colours in $\{0,1,2\}$,
and, for each marked face $f$ of $Q$, assign $x_f$ the same colour as $v_f$.  For each $i\in\{0,1,2\}$ define $X_i$ as the set of vertices of $H$ that are assigned colour $i$.

Since we start with a proper colouring of $M$, which contains $Q$, the sets $X_0,X_1,X_2$ satisfy \ref{proper}.  For each $x\in I(H)$, $x=x_f$ for some marked face $f$ of $Q$.  Then $x$ and $v_f$ have the same colour, so the sets $X_0,X_1,X_2$ satisfy \ref{inner_attachment}. Finally, we must show that each $X_i$ dominates $V(H)$ for each $i\in\{0,1,2\}$.  

First, consider a vertex $v\in B(H)=V(Q)$.  If $v$ is on the boundary of an unmarked face $g$ of $Q$, then $g$ is a triangle, and its three vertices receive three distinct colours.  Thus $N_H[v]$ includes a vertex in each of $X_0$, $X_1$, and $X_2$. If $v$ is not on the boundary of any unmarked face of $Q$, then $v$ is on the boundary of a marked face $f$ of $Q$.  Furthermore, $v\neq v_f$ since $v_f$ is the endpoint of a chord of $Q$, so $v_f$ is also on the boundary of another face $g$ of $Q$ which, by \cref{critical_structure}, is unmarked.  The inner vertex $x_f$ is in $X_i$ for some $i\in\{0,1,2\}$, and $v\in X_j$ for some $j\in\{0,1,2\}\setminus\{i\}$. The vertex $v$ has two neighbours in $Q$, both of which are on the boundary of $f$ and at least one of which, say $w$, is not $v_f$.  Since $vw$ and $wx_f$ are edges of $M$, $w\in X_{k}$ for $k\in \{0,1,2\}\setminus\{i,j\}$.  Thus, $\{i,j,k\}=\{0,1,2\}$, so $N_H[v]$ includes a vertex in each of $X_0$, $X_1$, and $X_2$.  

All that remains is to consider a vertex $x\in I(H)$.  In this case, $x=x_f$ for some face $f$ of $Q$.  The face $f$ has at least three vertices on its boundary.  Then $f-v_f$ is a path with at least one edge $v_1v_2$.  Then $x\in X_i$ for some $i\in\{0,1,2\}$, $v_1\in X_j$ for some $j\in\{0,1,2\}\setminus\{i\}$, and $v_2\in X_k$ for some $k\in\{0,1,2\}\setminus\{i,j\}$.  Thus $\{i,j,k\}=\{0,1,2\}$, so $N_H[x]$ includes a vertex in each of $X_0$, $X_1$, and $X_2$.  Hence each of $X_0$, $X_1$, and $X_2$ dominates $H$, which satisfies \ref{dominates_h}.
\end{proof}

The following lemma, illustrated in \cref{even_wheel}
 explains how we deal with even wheels not covered by \cref{biconnected_critical}:

\begin{lem}\label{wheelie}
  Let $H:=W_k$ for some even integer $k\ge 4$ and let $v$ be any vertex in $B(H)$.  Then there exists a partition $\{X_0,X_1,X_2\}$ of $V(H)$ such that
  \begin{enumerate}[label=\rm(\roman*),nosep,nolistsep]
    \item for each edge $ab$ of $H[B(H)]$, $a\in X_i$ and $b\in X_j$ for some $i\neq j$;\label{proper2}
    
    \item $v\in X_0$, $X_0$ and $X_2$ each dominate $H$, and $X_1$ dominates $V(H)\setminus\{v\}$; and \label{weak_dominates_h}
    \item each $x\in I(H)\cap X_i$ has a neighbour in $B(H)\cap X_i$. \label{inner_attachment_even}
  \end{enumerate}
\end{lem}

\begin{figure}
  \centering
  \includegraphics{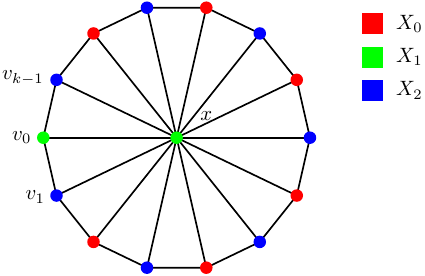}
  \caption{\cref{wheelie}: Partitioning the vertices of an even wheel into sets $X_0$, $X_1$, and $X_2$.}
  \label{even_wheel}
  \end{figure}
\begin{proof}
Let $v_0,\ldots,v_{k-1}$ be the vertices in $B(H)$ in clockwise order when traversing the outer face of $H$, with $v=v_0$. Let $x$ be the single vertex in $I(H)$.
Take $X_0:=\{v_0,x\}$, $X_1:=\{v_2,v_4,\ldots,v_{k-2}\}$, and $X_2:=\{v_1,v_3,\ldots,v_{k-1}\}$.  It is straightforward to check that $\{X_0,X_1,X_2\}$ is a vertex-partition of $H$ satisfying the requirements of the lemma.
\end{proof}

The following lemma, illustrated in \cref{critical_colouring}, drops the requirement that the critical graph be biconnected and applies even if some of the biconnected components of $H$ are even wheels.

\begin{figure}
  \centering
  \includegraphics[page=2, scale=0.7]{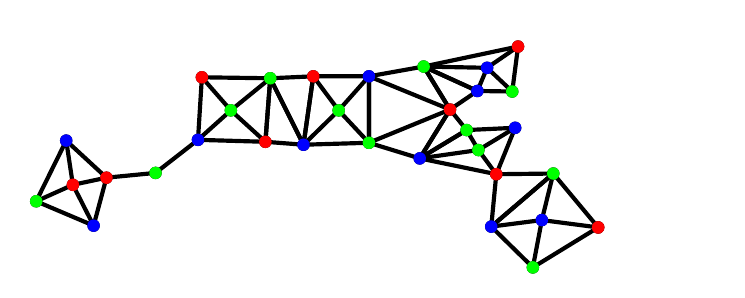}
  \caption{\cref{critical}: Partitioning the vertices of a connected critical graph into dominating sets $X_0$, $X_1$, and $X_2$.}
  \label{critical_colouring}
\end{figure}

\begin{lem}\label{critical}
  Let $H$ be a connected critical generalized near-triangulation with at least $3$ vertices, no vertices of degree $1$, and not isomorphic to $W_k$ for any even integer $k$.  Then there exists a partition $\{X_0,X_1,X_2\}$ of $V(H)$ such that
  \begin{enumerate}[label=\rm(\roman*),nosep,nolistsep]
    \item for each edge $vw$ of $H[B(H)]$, $v\in X_i$ and $w\in X_j$ for some $i\neq j$;\label{proper_2}
    \item for each $i\in\{0,1,2\}$, $X_i$ dominates $H$; and \label{dominates_h_minus_l}
    \item each $x\in I(H)\cap X_i$ has a neighbour in $B(H)\cap X_i$. \label{inner_attachment_connected}
  \end{enumerate}
\end{lem}

\begin{figure}[htbp]
    \centering
    \begin{tabular}{cc}
      \includegraphics[page=1,trim={30 0 30 0},clip]{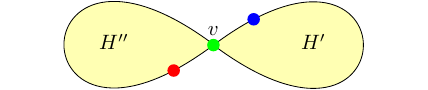} &
      \includegraphics[page=2]{figs/critical_3_colouring}
    \end{tabular}
    \caption{Two cases in the proof of \cref{critical}.}
    \label{critical_3_colouring}
  \end{figure}

\begin{proof}
  The proof is by induction on $|V(H)|$. If $H$ is biconnected, the lemma follows immediately from \cref{biconnected_critical}; this includes the case in which $|V(H)|=3$.

  If $H$ has more than three vertices and is not biconnected, then let $A,B\subseteq V(H)$ such that $A\cup B=V(H)$, $H[A]$ is biconnected, no edge joins $A\setminus\{v\}$ and $B\setminus\{v\}$, and $\{v\}:=A\cap B$ has size $1$.  Let $H':=H[A]$.\footnote{In the language of biconnectivity, $H'$ corresponds to a leaf in the block-cut tree of $H$.}  Then $H'$ is a biconnected induced subgraph of $H$ that contains $v$ and has fewer vertices than $H$. Since $H$ has no vertices of degree $1$, $H'$ has at least three vertices.  

  If $H'$ is not an even wheel, then apply \cref{biconnected_critical} to obtain a vertex partition $\{X_0',X_1',X_2'\}$ of $H'$ in which each part dominates $H'$.  Without loss of generality (by renaming the sets $X_0'$, $X_1'$, and $X_2'$) we may assume that $v\in X_0'$.  If $H'$ is an even wheel, then apply \cref{wheelie} to obtain a vertex partition $\{X_0',X_1',X_2'\}$ of $H'$ with $v\in X_0'$ such that each of $X_0'$ and $X_2'$ dominates $H'$ and $X_1'$ dominates $H'-v$.
  
  Let $H''=H-(V(H')\setminus\{v\})$. If $\deg_{H''}(v)\ge2$, then $H''$ is connected and critical, has minimum degree at least two, and has fewer vertices than $H$.  If $H''$ is not an even wheel, then apply induction to obtain a vertex-partition $\{X_0'',X_1'',X_2''\}$ of $H''$ satisfying the conditions of the lemma, with $v\in X_0''$.  If $H''$ is an even wheel, then apply \cref{wheelie} to obtain a vertex partition $\{X_0'',X_1'',X_2''\}$ of $H''$ such that each of $X_0''$ and $X_2''$ dominates $H''$ and $X_1''$ dominates $H''-v$ with $v\in X_0''$.  Let $X_0:=X_0'\cup X_0''$, $X_1:=X_1'\cup X_2''$, and $X_2:=X_2'\cup X_1''$.  Since $v\in X_0'\cap X_0''$ and $H'-v$ and $H''-v$ are vertex-disjoint, $\{X_0,X_1,X_2\}$ is a vertex-partition of $H$. That this partition satisfies \ref{proper_2} and \ref{inner_attachment_connected} is due to the fact that $\{X_0',X_1',X_2'\}$ and $\{X_0'',X_1'',X_2''\}$ satisfy these conditions for $H'$ and $H''$, respectively. For the same reason, each $X_0$, $X_1$, and $X_2$ dominates $V(H)\setminus\{v\}$.  That these sets also dominate $v$ comes from the fact that $v\in X_0$, $N_H(v)\supseteq N_{H'}(v)$ contains a vertex in $X_2'\subseteq X_2$, and $N_H(v)\supseteq N_{H''}(v)$ contains a vertex in $X_2''\subseteq X_1$.  Thus $\{X_0,X_1,X_2\}$ satisfies \ref{dominates_h_minus_l}.
  
  If $\deg_{H''}(v)=1$, then let $v_0:=v$ and $v_0,\ldots,v_{r-1},v_r$ be the maximal path in $H$ such that $\deg_H(v_i)=2$ for each $i\in\{1,\ldots,r-1\}$.  Since $H$ has no vertex of degree $1$, $\deg_H(v_r)=\deg_{H''}(v_r)\ge 3$.  Then $H''':=H''-\{v,v_1,\ldots,v_{r-1}\}$ is connected, critical, has minimum degree at least two, and has fewer vertices than $H$.  As in the preceding case, apply induction or \cref{wheelie} to $H'''$ to obtain a vertex partition $\{X_0''',X_1''',X_2'''\}$ of $H'''$.  By renaming, we can assume that $v_r\in X_{r\bmod 3}$, that $X_{r\bmod 3}$ and $X_{(r+1)\bmod 3}$ each dominate $H'''$, and that $X_{(r-1)  \bmod 3}$ dominates $V(H''')\setminus\{v_r\}$.  For each $i\in\{0,1,2\}$, define
  \[
      X_i := X_i' \cup \{ v_j:j\in\{1,\ldots,r-1\},\, j\equiv i\pmod 3\} \cup X_i''' \enspace .  
  \]
  Then $\{X_0,X_1,X_2\}$ satisfies \ref{proper_2} because $\{X_0',X_1',X_2'\}$ and $\{X_0''',X_1''',X_2'''\}$ satisfy \ref{proper_2} for $H'$ and $H'''$ respectively and because, for each $j\in\{1,\ldots,r-1\}$, $v_j\in X_{j\bmod 3}$ and the two neighbours of $v_j$ are $v_{j-1}\in X_{(j-1)\bmod 3}$ and $v_{j+1}\in X_{(j+1)\bmod 3}$.  This partition satisfies \ref{inner_attachment_connected} because $\{X_0',X_1',X_2'\}$ and $\{X_0''',X_1''',X_2'''\}$ satisfy \ref{inner_attachment_connected} for $H'$ and $H'''$ respectively, and $v_1,\ldots,v_{r-1}$ are in $B(H)$.  All that remains is to verify that $\{X_0,X_1,X_2\}$ satisfies \ref{dominates_h_minus_l}. For each vertex $w$ in $H'-v$, $N_{H'}(w)\subseteq N_{H}(w)$ contains a vertex in each of $X_0'\subseteq X_0$, $X_1'\subseteq X_1$, and $X_2'\subseteq X_2$.  For each vertex $w$ in $H'''-v_r$, $N_{H'''}(w)\subseteq N_{H}(w)$ contains a vertex in each of $X_0'''\subseteq X_0$, $X_1'''\subseteq X_1$, and $X_2'''\subseteq X_2$. 
  For each $j\in\{1,\ldots,r-1\}$, $N_H[v_j]$ contains a vertex in each of $X_{(j-1)\bmod 3}$, $X_{j\bmod 3}$, and $X_{(j+1)\bmod 3}$. The vertex $v=v_0$ is in $X_0$, $N_{H'}(v)\subseteq N_H(v)$ contains a vertex in $X_2'\subseteq X_2$, and $N_H(v)$ contains the vertex $v_1\in X_1$.  The vertex $v_r$ is in $X_{r\bmod 3}$, $N_{H'''}(v_r)\subseteq N_H(v_r)$ contains a vertex in $X_{(r+1)\bmod 3}'''\subseteq X_{(r+1)\bmod 3}$, and $N_H(v_r)$ contains the vertex $v_{r-1}\in X_{(r-1)\bmod 3}$.
\end{proof}

At last, the following lemma, illustrated in \cref{two_critical_colouring}, shows how we combine everything to find three sets whose total size is at most $2|B(H-B(H))| + |I(H-B(H))|$.

\begin{figure}
  \centering
  \includegraphics[page=4, scale=0.7]{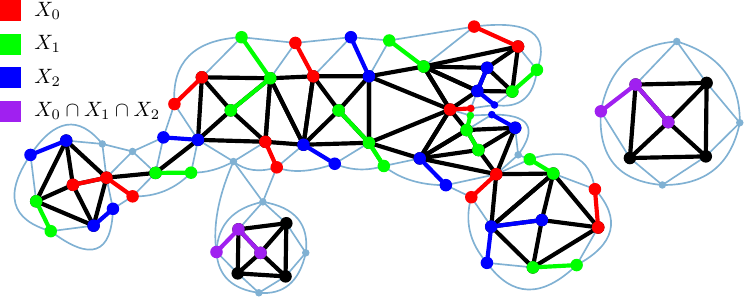}
  \caption{\cref{three_sets_coverage}: Finding three sets $X_0$, $X_1$, and $X_2$ that dominate $I(H)$ in a $2$-critical graph $H$.}
  \label{two_critical_colouring}
\end{figure}

\begin{lem}\label{three_sets_coverage}
  Let $H$ be a $2$-critical generalized near-triangulation. Then there exist $X_0,X_1,X_2\subseteq V(H)$ such that
  \begin{enumerate}[label=\rm(\roman*),nosep,nolistsep]
    \item $|X_0|+|X_1|+|X_2| \le 2|B(H-B(H))|+|I(H-B(H))|$; \label{total_size}
    \item for each $i\in\{0,1,2\}$, $X_i$ dominates $I(H)$ in $H$; and \label{dominates_i}
    \item for each $i\in\{0,1,2\}$, each component of $H[X_i]$ contains at least one vertex in $B(H)$. \label{connectivity}
  \end{enumerate}
\end{lem}

\begin{proof}
  Let $\mathcal{C}$ be the set of components of $H-B(H)$. By \ref{inner_degree_2} and \ref{inner_critical}, each $C\in\mathcal{C}$ is critical and has minimum degree at least two.  For each $w\in B(C)$, $N_H(w)$ includes a vertex in $B(H)$.  Initialize $X_0:=X_1:=X_2:=\emptyset$, and process each component $C\in\mathcal{C}$ in turn.

If $C$ is an even wheel, then let $\{x\}:=I(C)$, let $v\in B(C)$, and let $w\in N_H(v)\cap B(H)$.  Add $\{x,v,w\}$ to $X_{i}$ for each $i\in\{0,1,2\}$.  This increases $|X_0|+|X_1|+|X_2|$ by $9=2\cdot 4+1\le 2|B(C)|+|I(C)|$.  Since $x$ is included in $X_i$, this ensures that $X_i$ dominates $I(H)\cap V(C)$ for each $i\in\{0,1,2\}$, so this ensures \ref{dominates_i}. The induced subgraph $H[\{x,v,w\}]$ is connected and contains $w\in B(H)$, which ensures \ref{connectivity}.

If $C$ is not an even wheel, then apply \cref{critical} to obtain a vertex partition $\{Y_0,Y_1,Y_2\}$ of $C$.  Add $Y_i$ to $X_i$ for each $i\in \{0,1,2\}$. This increases $|X_0|+|X_1|+|X_2|$ by $|Y_0|+|Y_1|+|Y_2| = |C| = |I(C)|+|B(C)|$.  For each $w\in Y_i\cap B(C)$, choose a neighbour $a(w)\in N_H(w)\cap B(H)$ and also include $a(w)$ in $X_i$.  This increases $|X_0|+|X_1|+|X_2|$ by at most $|B(C)|$. Thus, this step increases $|X_0|+|X_1|+|X_2|$ by at most $2|B(C)|+|I(C)|$.  By \cref{critical} property \ref{dominates_h_minus_l}, $X_i$ dominates $V(C)=I(H)\cap V(C)$, which ensures \ref{dominates_i}.  By \ref{inner_attachment_connected}, each component of $C[Y_i]$ contains a vertex $w\in B(C)$ and the addition of $a(w)$ to $X_i$ ensures that each component of $H[X_i]$ contains a vertex of $B(H)$, which ensures \ref{connectivity}.

For each component $C\in\mathcal{C}$, processing $C$ increases $|X_0|+|X_1|+|X_2|$ by at most $2|B(C)|+|I(C)|$.  Thus, $|X_0|+|X_1|+|X_2| \le \sum_{C\in\mathcal{C}}(2|B(C)|+|I(C)|) = 2|B(H-B(H))| + |I(H-B(H))|$, which ensures \ref{total_size}.
\end{proof}

\begin{proof}[Proof of \cref{two_critical_handler}]
  Take $X$ to be the smallest of the three sets $X_0$, $X_1$, and $X_2$ guaranteed by \cref{three_sets_coverage}.
\end{proof}

\subsection{The Algorithm}

\emergencystretch=2em
We now define $\textsc{BetterGreedy}(G)$, a refinement of $\textsc{SimpleGreedy}(G)$.  If $G$ has $n=3$ vertices, then it has a connected dominating set of size $1=n/3<10n/21$, so we may assume that $n\ge 4$.

Let $H_0$ be a dom-minimal generalized near-triangulation that dom-respects $G$ (which is guaranteed to exist, by \cref{dom_minimal}). As in $\textsc{SimpleGreedy}(G)$, we construct a sequence $\Delta_0,\Delta_1,\ldots,\Delta_{r}$ of subsets of $V(G)$. This defines a nested sequence of sets $X_1\subseteq\cdots\subseteq X_{r+1}$ with $X_i:=\bigcup_{j=0}^{i-1}\Delta_j$ for each $i\in\{1,\ldots,r+1\}$.  However, for each $i\in\{0,\ldots,r\}$, the graph we use to compute $\Delta_i$ is not $G_i=G-X_i$.  Ideally, we would always choose $\Delta_i=\{v_i\}$ where $v_i\in B(G_i)$ has $\deg^+_{G_i}(v_i)\ge 3$.  However, unlike the critical graphs that appear when $\deg^+_G(v)\le 1$ for all $v\in B(G)$, there are near-triangulations $G$ with $\deg^+_G(v)=2$ for all $v\in B(G)$ for which $G[I(G)]$ is arbitrarily complicated; in particular, $|I(G)|$ cannot be bounded by any function of $|B(G)|$. To help deal with this, we work with dom-minimal graphs: For each $i\in\{1,\ldots,r\}$, $H_i$ is a dom-minimal graph that dom-respects $H_{i-1}-\Delta_{i-1}$. The fact that each $H_i$ is dom-minimal allows us to use \cref{two_three_killer,leaf_killer,degree_zero_killer} to deal with cases where $\deg^+_{H_i}(v)\le 2$ for all $v\in B(H_i)$. When none of these tools can be used, we are either left with a graph $H_r$ with $I(H_r)=\emptyset$ or a graph $H_r$ that is $2$-critical. In the former case, we can take $\Delta_r=\emptyset$. In the latter case, we use \cref{two_critical_handler} to compute the final set $\Delta_r$ that dominates the remaining vertices in $I(H_r)$.

During iteration $i\ge 0$, let $J_i:=H_i-B(H_i)$. We use the following procedure to choose $\Delta_i$:
\begin{enumerate}[label={[bg\arabic*]},labelwidth=\widthof{[bg0]},leftmargin=!,nosep,nolistsep] 
    \item If $I(H_i)=\emptyset$, then this is the final iteration, so $r:=i$. We set $\Delta_r:=\emptyset$ and return $X:=X_{r+1}:=X_r$.\label{bg_empty}
    \item Otherwise, if some $v_i\in B(H_i)$ has $\deg^+_{H_i}(v_i)\ge3$, choose any such vertex and set $\Delta_i:=\{v_i\}$.  
    \label{bg_high_degree}
    \item Otherwise, if $J_i$ is not critical, set $\Delta_i:=\{v,w\}$, where $v$ and $w$ are the vertices guaranteed by \cref{two_three_killer}.  
    \label{bg_two_three}
    \item Otherwise, if $J_i$ has a vertex of degree~1, set $\Delta_i:=\{v_i\}$, where $v_i$ is the vertex guaranteed by \cref{leaf_killer}.
    \label{bg_leaf}
    \item Otherwise, if $J_i$ has a vertex of degree~0, set $\Delta_i:=\{v_i\}$, where $v_i$ is the vertex guaranteed by \cref{degree_zero_killer}.  
    \label{bg_degree_zero}
    \item Otherwise, $H_i$ is $2$-critical. This is the final iteration, so $r:=i$. Use \cref{two_critical_handler} to choose $\Delta_r$ of size at most
    \[
      \frac{2|B(J_i)|+|I(J_i)|}{3}
    \]
    that dominates $I(H_i)$ in $H_i$ and such that every component of $H_i[\Delta_r]$ contains a vertex in $B(H_i)$.  Return $X:=X_{r+1}:=X_r\cup\Delta_r$.
    \label{bg_two_critical}
\end{enumerate}

\begin{thm}\label{better_greedy}
  When applied to an $n$-vertex triangulation $G$ with $n\ge4$, $\textnormal{\textsc{BetterGreedy}}(G)$ produces a connected dominating set $X$ of size at most $(10n-18)/21$.  Additionally, $X$ contains at least one vertex on the outer face of $G$. 
\end{thm}
\begin{proof}
  That $X$ contains a vertex on the outer face of $G$ follows immediately from the fact that $B(H_0)\subseteq B(G)$ and $\Delta_0\subseteq X$ contains a vertex in $B(H_0)$.
  Next we argue that $X$ is a connected dominating set of $G$.  This is based on the following claim: For each $i\in\{1,\ldots,r+1\}$ and each $v\in B(H_{i-1}-\Delta_{i-1})$, $N_G(v)$ contains a vertex in $X_i$.  We now prove this claim.  If $v\in B(G)$, then $N_G(v)$ contains a vertex in $\Delta_0=X_1\subseteq X_i$, since $\Delta_0$ contains a vertex in $B(H_0)\subseteq B(G)$, $N_G[v]\supseteq B(G)$, and $v\not\in X_i$. Otherwise, $v\in B(H_{i-1}-\Delta_{i-1})\setminus B(G)$.  In this case, let $j\in\{1,\ldots,i\}$ be the smallest index such that $v\in B(H_{j-1}-\Delta_{j-1})$.  Since $v\not\in B(G)=B(H_0)$, $v\in I(H_0)$, and by the minimality of $j$, $v\in I(H_{j-1})$. In particular, $v\not\in B(H_{j-1})$, so $N_{H_{j-1}}(v)$ contains a vertex in $\Delta_{j-1} \subseteq X_j\subseteq X_i$.  Since $v\in I(H_\ell)$ for $\ell\in\{0,\ldots,j-1\}$, \ref{same_neighbourhood} and \ref{same_inner_vertices} imply that $N_{H_{j-1}}(v)\subseteq N_{H_{j-2}}(v)\subseteq\cdots\subseteq N_{H_{0}}(v)\subseteq N_{G}(v)$. Thus, $N_G(v)$ contains a vertex in $\Delta_{j-1}\subseteq X_i$, as claimed.

  We now use this claim to show that $X=X_{r+1}$ is a dominating set in $G$. Let $v$ be a vertex of $G$.  If $v\in X_{r+1}$ then there is nothing to prove, so suppose that $v\not\in X_{r+1}$.  
  If $v\in B(G)$, then $N_G(v)$ contains a vertex in $\Delta_0=X_1\subseteq X_{r+1}$. Otherwise, $v\in I(G)=I(H_0)$.  We consider two cases:
  \begin{itemize}
    \item If $v\in I(H_r)$, then the algorithm must have terminated in case~\ref{bg_two_critical} (since $I(H_r)=\emptyset$ in case~\ref{bg_empty}). In case~\ref{bg_two_critical}, $\Delta_r$ dominates $I(H_r)$ in $H_r$ by \cref{two_critical_handler}. Thus, $v$ has a neighbour $w\in\Delta_r$ in $H_r$. Since $v\in I(H_r)$, \ref{same_neighbourhood} and \ref{same_inner_vertices} ensure that $vw\in E(G)$, so $w\in N_G(v)\cap X_{r+1}$.
    \item If $v\not\in I(H_r)$, let $i\in\{1,\ldots,r\}$ be the maximum index such that $v\in I(H_{i-1})$. Since $H_i$ dom-respects $H_{i-1}-\Delta_{i-1}$, \ref{same_inner_vertices} implies that $I(H_i)=I(H_{i-1}-\Delta_{i-1})$, so $v\not\in I(H_{i-1}-\Delta_{i-1})$. Since $v\in I(H_{i-1})$ and $v\not\in X_{r+1}$ (so $v\not\in\Delta_{i-1}$), we have $v\in V(H_{i-1}-\Delta_{i-1})\setminus I(H_{i-1}-\Delta_{i-1}) = B(H_{i-1}-\Delta_{i-1})$. The preceding claim establishes that $N_G(v)$ contains a vertex in $X_i\subseteq X_{r+1}$.
  \end{itemize}
  
  We now argue that $G[X_{r+1}]$ is connected.  For each $i\in\{0,\ldots,r\}$, each component of $H_i[\Delta_i]$ contains a vertex in $B(H_i)$. Let $v$ be a vertex in $X_{r+1}$, so $v\in \Delta_i$ for some $i\in\{0,\ldots,r\}$.  Then the component $C$ of $H_i[\Delta_i]$ that contains $v$ has a vertex in $B(H_i)$.  Let $P$ be a shortest path in $C$ from $v$ to a vertex $v_0\in B(H_i)$. By \ref{same_inner_vertices} and \ref{same_neighbourhood}, $P$ is a subgraph of $H_{i-1}-\Delta_{i-1}$, which is a subgraph of $H_{i-1}$. Repeating this argument shows that $P$ is a subgraph of $H_{i-2},\ldots,H_0$, and of $G$.  Therefore, $P$ is a path in $G[X_{r+1}]$.  If $i=0$, then $v_0\in B(H_0)\cap \Delta_0 \subseteq B(G)\cap X_{r+1}$.  If $i\ge 1$, then by \ref{boundary_subset}, $v_0\in B(H_i)\subseteq B(H_{i-1}-\Delta_{i-1})$. By the claim above, $N_G(v_0)$ contains a vertex $v_1$ in $X_i$, so $v_1\in \Delta_j$ for some $j\in\{0,\ldots,i-1\}$. Since $v_0\in\Delta_i\subseteq X_{r+1}$ and $v_1\in X_i\subseteq X_{r+1}$, the edge $v_0v_1$ is an edge of $G[X_{r+1}]$. Iterating this argument (or by induction on $i$) shows that $G[X_{r+1}]$ contains a path from $v$ to a vertex in $B(H_0)\cap \Delta_0 \subseteq B(G)\cap X_{r+1}$.  This implies that, for any two vertices $v$ and $w$ in $X_{r+1}$, $G[X_{r+1}]$ contains a path from $v$ to a vertex $v'\in B(G)\cap X_{r+1}$ and a path from $w$ to a vertex $w'\in B(G)\cap X_{r+1}$.  If $v'\ne w'$, then since $G[B(G)]$ is a clique, $v'w'\in E(G)$; since $v',w'\in X_{r+1}$, the edge $v'w'$ is in $G[X_{r+1}]$. This implies that $G[X_{r+1}]$ is connected.

  What remains is to analyze the size of $X_{r+1}$.
  We do this using a potential function that uses the following measure of a generalized near-triangulation $H$: 
  \[
     \Phi(H):=\frac{10|I(H)|+4|B(H)|}{21}.
  \]
  If a generalized near-triangulation $H$ dom-respects a generalized near-triangulation $H'$, then $I(H)=I(H')$ (by \ref{same_inner_vertices}) and $B(H)\subseteq B(H')$ (by \ref{boundary_subset}). Therefore, $\Phi(H)\le \Phi(H')$.
  For each $i\in\{1,\ldots,r\}$, the graph $H_i$ dom-respects $H_{i-1}-\Delta_{i-1}$, so $\Phi(H_i)\le\Phi(H_{i-1}-\Delta_{i-1})$.  
  For each $i\in\{0,\ldots,r\}$, let $\Phi_i:=\Phi(H_i)$.  Since $G$ is an $n$-vertex triangulation and $H_0$ dom-respects $G$,
  \[
    \Phi_0 = \Phi(H_0) \le \Phi(G) 
      = \frac{10|I(G)|+4|B(G)|}{21} 
      = \frac{10(n-3)+4(3)}{21} = \frac{10n-18}{21} \enspace .
  \]
  We will show that, for each $i\in\{1,\ldots,r\}$,
  \begin{equation}
     \Phi_{i-1}-\Phi_{i} \ge |\Delta_{i-1}|, \label{potential_drop}
  \end{equation}
  and that
  \begin{equation}
    \Phi_r \ge |\Delta_r| \enspace . \label{last_potential}
  \end{equation}
  Then
  \[
      |X_{r+1}| = \sum_{i=0}^r|\Delta_i| 
      \le \sum_{i=1}^{r}(\Phi_{i-1}-\Phi_i)+\Phi_r 
      =  \Phi_0 \le 
      \frac{10n-18}{21} \enspace .  
  \]
  All that remains is to establish \eqref{potential_drop} and \eqref{last_potential}.  We begin with \eqref{potential_drop}, which has four cases corresponding to cases \ref{bg_high_degree}, \ref{bg_two_three}, \ref{bg_leaf}, and \ref{bg_degree_zero} in the choice of $\Delta_{i-1}$.
  
  \begin{enumerate}[label={[bg\arabic*]},labelwidth=\widthof{[bg0]},leftmargin=!,nosep,nolistsep,start=2]
  \item In this case, $\Delta_{i-1}=\{v\}$ where $v\in B(H_{i-1})$ has inner degree $t:=\deg^+_{H_{i-1}}(v)\ge 3$. Then $|I(H_{i-1}-\Delta_{i-1})|=|I(H_{i-1})|-t$ and $|B(H_{i-1}-\Delta_{i-1})|=|B(H_{i-1})|+t-1$.  Since $t\ge 3$, this gives
  \[
    \Phi_{i-1}-\Phi_{i}\ge\Phi_{i-1}-\Phi(H_{i-1}-\Delta_{i-1})\ge\frac{10t-4(t-1)}{21}
    =\frac{6t+4}{21}\ge 1 = |\Delta_{i-1}|.
  \]
  \item In this case, $\Delta_{i-1}=\{v,w\}$ where $v$ and $w$ are chosen by \cref{two_three_killer} to ensure that $t:=\deg^+_{H_{i-1}-v}(w)\ge 3$ and the dom-minimal graph $H_i$ that dom-respects $H_{i-1}-\Delta_{i-1}$ has $|I(H_i)|=|I(H_{i-1})|-(t+2)$ and $|B(H_i)|\le |B(H_{i-1})|+t-1$.  Since $t\ge 3$, this gives
  \[
     \Phi_{i-1}-\Phi_i \ge \frac{10(t+2)-4(t-1)}{21} = \frac{6t+24}{21} \ge 2 = |\Delta_{i-1}|.
  \]
  \item In this case, $\Delta_{i-1}=\{v\}$ where $v$ is chosen by \cref{leaf_killer} to ensure that the dom-minimal graph $H_i$ that dom-respects $H_{i-1}-\Delta_{i-1}$ has $|I(H_{i})|=|I(H_{i-1})|-2$ and $|B(H_i)|\le |B(H_{i-1})|-1$.  Thus, $\Phi_{i-1}-\Phi_i\ge (10(2)+4(1))/21 = 24/21 > 1=|\Delta_{i-1}|$.
  \item In this case, $\Delta_{i-1}=\{v\}$ where $v$ is chosen by \cref{degree_zero_killer} to ensure that the dom-minimal graph $H_i$ that dom-respects $H_{i-1}-\Delta_{i-1}$ satisfies one of the following:
  \begin{itemize}
     \item $|I(H_i)|=|I(H_{i-1})|-2$ and $|B(H_i)|\le |B(H_{i-1})|-1$. In this case, $\Phi_{i-1}-\Phi_i\ge(10(2)+4(1))/21 = 24/21 > 1=|\Delta_{i-1}|$.
     \item $|I(H_i)|=|I(H_{i-1})|-1$ and $|B(H_i)|\le |B(H_{i-1})|-3$. In this case, $\Phi_{i-1}-\Phi_i\ge(10(1)+4(3))/21 = 22/21 > 1=|\Delta_{i-1}|$.
  \end{itemize} 
\end{enumerate}
To finish the proof, we need only establish \eqref{last_potential}: $\Phi_r \ge |\Delta_r|$.  The algorithm ends with $\Delta_r$ when \ref{bg_empty} or \ref{bg_two_critical} occurs.  When \ref{bg_empty} occurs, $\Delta_r=\emptyset$, so $\Phi_r\ge 0=|\Delta_r|$.  When \ref{bg_two_critical} occurs, $\Delta_r$ is chosen using \cref{two_critical_handler} so that $|\Delta_r|\le (2|B(J_r)|+|I(J_r)|)/3$, where $J_r:=H_r-B(H_r)$.  By definition, $I(H_r)=V(J_r)=B(J_r)\cup I(J_r)$.  By \cref{two_critical_boundary_size}, $|B(H_r)|\ge |B(J_r)|$.  Therefore,
\begin{align*}
   \Phi_r 
   & = \frac{10|I(H_r)|+4|B(H_r)|}{21} \\
   & \ge \frac{10(|B(J_r)|+|I(J_r)|) + 4|B(J_r)|}{21} \\
   & = \frac{14|B(J_r)|+10|I(J_r)|}{21} \ge \frac{2|B(J_r)|+|I(J_r)|}{3}\ge |\Delta_r| \enspace . \qedhere
\end{align*}
\end{proof}

\section{Connected Domination in Surface Triangulations}
\label{bounded_genus}

In this section, we establish \cref{genus_result}, the extension of \cref{main_result2} to surface triangulations of Euler genus $g$.
Briefly, a \defin{surface} $\mathcal{S}$ is a compact connected Hausdorff topological space that is locally homeomorphic to the plane, i.e., every point in $\mathcal{S}$ has a neighbourhood homeomorphic to $\mathbb{R}^2$. Every such surface can be created from the sphere $\mathbb{S}^2$ by adding handles or cross-caps. The \defin{Euler genus} of a surface with $h$ handles and $c$ cross-caps is $2h+c$.

We follow the definitions given in \citet[Appendix~B]{diestel:graph}. An \defin{arc}, a \defin{circle}, and a \defin{disc} in a surface $\mathcal{S}$ are subsets of $\mathcal{S}$ that are homeomorphic to $[0, 1]$, the unit circle $\mathbb{S}^1 = \{x \in \mathbb{R}^2: \|x\| = 1\}$, and the open unit disc $\mathbb{B}^2 = \{x \in \mathbb{R}^2 : \|x\| < 1\}$, respectively.
The set of all arcs in $\mathcal{S}$ is denoted by $A_\mathcal{S}$.
An \defin{embedding} of a graph $G$ in a surface $\mathcal{S}$ is a map $\sigma: V(G)\cup E(G) \longrightarrow \mathcal{S}\cup A_\mathcal{S}$ that sends vertices of $G$ to distinct points in $\mathcal{S}$ and sends an edge $xy$ to
an arc $\sigma(xy)$ in $A_\mathcal{S}$ with endpoints $\sigma(x)$ and $\sigma(y)$ in such a way that the interior of $\sigma(xy)$ is disjoint from $\{\sigma(v):v\in V(G)\}$ and from the interior of $\sigma(vw)$ for every $vw\in E(G)\setminus\{xy\}$. For any subgraph $H\subseteq G$, we write $\sigma(H):=\{\sigma(v):v\in V(H)\}\cup \bigcup_{vw\in E(H)}\sigma(vw)$.
A graph $G$ equipped with an embedding $\sigma$ on a surface $\mathcal{S}$ is called an \defin{$\mathcal{S}$-embedded} graph. A \defin{face} of $G$ in $\mathcal{S}$ is a component of $\mathcal{S} \setminus \sigma(G)$.

The \defin{surface map} $\Sigma$ of an $\mathcal{S}$-embedded graph $G$ is a tuple $(V, E, F)$ where $V$ is the set of vertices, $E$ is the set of edges, and $F$ is the set of faces in the embedding $\sigma$ of $G$. If each face of $G$ is homeomorphic to an open disc in $\mathcal{S}$, then each edge in $E$ is on the boundary of exactly two faces in $F$ and gluing the faces in $F$ along their shared edges yields the original surface $\mathcal{S}$.  We call a surface map a \defin{surface triangulation} if every face in $F$ is a disc in $\mathcal{S}$ whose boundary is an embedded $3$-cycle of $G$.  The \defin{Euler genus} of a surface triangulation $\Sigma$ is the Euler genus of the surface $\mathcal{S}$.

As in \cite{EricksonNotes}, \defin{slicing} the surface map $\Sigma$ of an $\mathcal{S}$-embedded graph $G$ along a subgraph $Z \subseteq G$ with at least one edge produces a new surface map $\Sigma \bbslash Z=(V',E',F')$ where $V'$ contains $\deg_{Z}(v)$ copies of every vertex $v\in V(G)$ with $\deg_Z(v)\ge 1$ and one copy of every vertex $v\in V(G)$ with $\deg_Z(v)=0$; $E'$ contains one copy of every edge in $E(G)\setminus E(Z)$, and two copies of every edge of $Z$. Each copy of an edge of $Z$ is incident to only one face in $F'$ (its other side forms a boundary edge of the sliced surface), whereas each edge in $E(G)\setminus E(Z)$ is incident to two faces in $F'$. Intuitively, $\Sigma\bbslash Z$ describes the surface with boundary $\mathcal{S}\bbslash Z$ obtained by gluing the faces in $F$ along every edge in $E(G)\setminus E(Z)$.  Each boundary component of $\mathcal{S}\bbslash Z$ is called a \defin{hole}. A \defin{planarizing subgraph} of $\Sigma$ is any subgraph $Z \subseteq G$ such that every connected component of $\mathcal{S}\bbslash Z$ has genus $0$ \cite{EricksonNotes}. A key property of the slicing operation is the following: Every edge $e'\in E'$ on the boundary of a hole corresponds to an edge $e\in E(Z)$. Every vertex $v'\in V'$ on the boundary of a hole corresponds to a vertex $v\in V(Z)$.  If vertices $v'$ and $w'$ are on the boundary of the same hole in $\Sigma\bbslash Z$, then the corresponding vertices $v$ and $w$ of $G$ are in the same component of $Z$. We use the following theorem of \citet{10.5555/644108.644208}.

\begin{figure}
    \begin{center}
        \begin{tabular}{cc}
            \includegraphics[width=.45\textwidth,page=2]{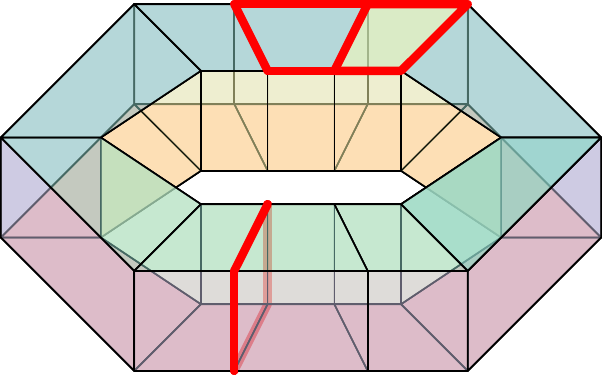} &
            \includegraphics[width=.45\textwidth,page=4]{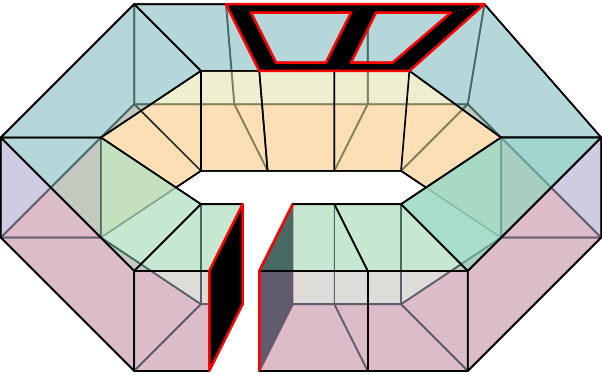} \\
            $\Sigma$ & $\Sigma\bbslash Z$
        \end{tabular}
    \end{center}
    \caption{Slicing a surface map $\Sigma$ along a subgraph $Z$ (in red). Holes in $\Sigma\bbslash Z$ are shown in black and hole boundaries are red.}
    \label{slicing}
\end{figure}

\begin{thm}[\citet{10.5555/644108.644208}]\label{Eppstein}
Every surface triangulation $\Sigma$ with $n$ vertices and Euler genus $1\le g < n$ has a planarizing subgraph $Z$ with $O(\sqrt{gn})$ vertices and edges.
\end{thm}

We can now prove the main result of this section, which readily establishes \cref{genus_result}.

\begin{lem}\label{generic_surface_thm}
  Suppose that for every $n\ge 3$, every $n$-vertex triangulation has a connected dominating set of size at most $\alpha n + \beta$ for some constants $\alpha, \beta>0$. Then, for every $1\le g < n$, every $n$-vertex surface triangulation of Euler genus $g$ has a connected dominating set of size at most $\alpha n + O(\sqrt{gn})$.
\end{lem}

\begin{proof}
  Let $\mathcal{S}$ be a surface of Euler genus $g$ and let $G$ be a graph embedded on $\mathcal{S}$ such that the resulting surface map $\Sigma:=(V,E,F)$ is a surface triangulation.
  By \cref{Eppstein}, $\Sigma$ has a planarizing subgraph $Z$ with $O(\sqrt{gn})$ vertices and edges.  If $Z$ contains a vertex $v$ with $\deg_Z(v)<2$, then $Z-v$ is also a planarizing subgraph. Repeatedly removing vertices of degree at most $1$ from $Z$ and their incident edges leaves a planarizing subgraph with minimum degree at least $2$. We may therefore assume that $\deg_Z(v)\ge 2$ for each $v\in V(Z)$.  
Since $g \ge 1$, $Z$ contains at least one non-contractible cycle, so $Z$ remains non-empty.

Let $\Sigma':=(V',E',F'):=\Sigma\bbslash Z$.  Let $G'$ be the graph with vertex set $V'$ and edge set $E'$.  Then $|V'|\le n + \sum_{v\in V(Z)}(\deg_Z(v)-1) \le n + O(\sqrt{gn})$. The total number of edges on the boundaries of the holes in $\Sigma'$ is $2|E(Z)|=O(\sqrt{gn})$.
For each component $C$ of $G'$, $\Sigma'$ describes an embedding of $C$ on a surface $\mathcal{S}_C$ (with boundary) of genus $0$.  Each boundary component of $\mathcal{S}_C$ corresponds to a hole in $\Sigma'$.  The boundary of each hole corresponds to a cycle in $C$, and these boundary cycles are pairwise vertex-disjoint.  By filling each hole with a disc, we can treat $C$ as a graph embedded on the sphere. Each disc we attach then becomes a \emph{holy} face of $C$.  Choose some holy face $f_0$ of $C$ and call $f_0$ the outer face.   For each holy face $f$ of $C$, introduce a new vertex $x_f$ in the interior of $f$ that is adjacent to all vertices of $C$ on the boundary of $f$. The addition of these vertices and edges to $C$ makes a simple planar triangulation $C^+$. Define the outer face of $C^+$ to be any face of $C^+$ contained in $f_0$. By \cref{better_greedy}, $C^+$ has a connected dominating set $X_C$ of size at most $\alpha|V(C^+)|+\beta$ that contains a vertex of the outer face.

Let $X':=\bigcup_{C} X_C$.  The set $X'$ contains vertices of $\Sigma'$ that may or may not appear in $\Sigma$. To obtain a connected dominating set $X$ for $\Sigma$, we set $X:=(X'\cap V(G))\cup V(Z)$. We now show that $X$ satisfies the requirements of the lemma.
\begin{itemize}
    \item \emph{$X$ is a dominating set:} Since $X'$ is a dominating set of $G'$, each vertex $w\in V(G)\setminus X$ is adjacent, in $G'$, to some vertex $v'\in X'$. Because $w \notin X \supseteq V(Z)$, $w$ is not on the boundary of any hole, so $w$ is not adjacent to $x_f$ for any holy face $f$. Thus, $v' \ne x_f$ for any holy face $f$. If $v'\in V(G)$, then $v'\in X' \cap V(G) \subseteq X$. If $v'\in V(G')\setminus V(G)$, then $v'$ is a copy of some vertex $v \in V(Z) \subseteq X$. In either case, $w$ has a neighbour in $X$. Thus, $X$ is a dominating set of $G$.

    \item \emph{$G[X]$ is connected:} Let $v$ and $w$ be any two vertices in $X$. Then each of $v$ and $w$ has at least one corresponding vertex $v'$ and $w'$, respectively, in $G'$. We first consider the case in which $v'$ and $w'$ are in the same component $C$ of $G'$. Let $F_C$ be the set of holy faces of $C$, let $A_C:=\{x_f:f\in F_C\}$, and let $Y_C:=\bigcup_{f\in F_C} N_{C^+}(x_f)$. Then $C^+[A_C\cup X_C\cup Y_C]$ is connected because $C^+[X_C]$ is connected and $X_C$ dominates $V(C)\supseteq Y_C$ in $C^+$. For each $x_f\in A_C$, $C[N_{C^+}(x_f)]=C[V(f)]$ is connected because $f$ is a facial cycle in $C$.  Furthermore, $A_C=V(C^+)\setminus V(C)$ and $N_{C^+}(x_f)\subseteq V(C)$ for each $x_f \in A_C$, so $A_C$ is an independent set in $C^+$. It follows that $H_C:=C[(X_C\cup Y_C)\setminus A_C]$ is connected.  Let $X_C'\subseteq V(G)$ be obtained by replacing each vertex $x'\in (X_C\cup Y_C)\setminus A_C$ with the corresponding vertex $x$ of $G$.  The graph $G_C:=G[X_C']$ can be obtained from $H_C$ by identifying vertices, which preserves connectivity.  Thus, $G_C$ is connected.
    Since $v, w \in X$, their corresponding vertices $v', w'$ in $C$ belong to $(X_C\cup Y_C)\setminus A_C$, so $v, w \in X_C'$. Therefore, $G_C \subseteq G[X]$ contains a path from $v$ to $w$.
    
    Now consider the case when $v'$ and $w'$ lie in different components of $G'$.  Let $Q$ be any path in $G$ from $v$ to $w$. Partition the edges of $Q$ into maximal subpaths $Q_1,Q_2,\dots,Q_k$ such that each $Q_i$ corresponds to a path that is contained in some component $C_i$ of $G'$.  For each $i\in\{2,\ldots,k-1\}$, the subpath $Q_i$ has endpoints $v_i,w_i\in V(Z)\subseteq X$ and the component $C_i$ has corresponding vertices $v_i',w_i'\in Y_{C_i}$. The subpath $Q_1$ has endpoints $v_1:=v\in X$ corresponding to $v_1'\in V(C_1)$ and $w_1\in V(Z)\subseteq X$ corresponding to $w_1'\in Y_{C_1}$.  Since $v_1\in X$, $v_1'\in X_{C_1}\cup Y_{C_1}$.  The subpath $Q_k$ has endpoints $v_k\in V(Z)$ corresponding to $v_k'\in Y_{C_k}$ and $w_k:=w\in X$ corresponding to $w_k'\in V(C_k)$.  Since $w_k\in X$, $w_k'\in X_{C_k}\cup Y_{C_k}$. Thus, for each $i\in\{1,\ldots,k\}$, $H_{C_i}$ contains a path from $v_i'$ to $w_i'$, so $G_{C_i}\subseteq G[X]$ contains a path from $v_i$ to $w_i$.  Therefore, $G[X]$ contains a path from $v=v_1$ to $w=w_k$.

    \item \emph{$X$ has size at most $\alpha n + O(\sqrt{gn})$:} Each hole $h$ in $\Sigma'$ corresponds to a cycle $D_h$ in $Z$ with at least three edges. Each edge of $Z$ participates in $D_h$ for two holes $h$. Therefore, the number of holes in $\Sigma'$ is at most $2|E(Z)|/3\in O(\sqrt{gn})$. Each hole in $\Sigma'$ contributes one holy face to one component of $G'$, and each component of $G'$ has at least one holy face. Therefore, $G'$ has $O(\sqrt{gn})$ components.  Each holy face $f$ of a component $C$ of $G'$ contributes one vertex $x_f$ to $V(C^+)$. Therefore, $\sum_{C}|V(C^+)|\le |V'| + O(\sqrt{gn}) = n + O(\sqrt{gn})$, and $|X|\le |V(Z)| + \sum_{C} (\alpha|V(C^+)|+\beta)\le \alpha n + O(\sqrt{gn})$. \qedhere
\end{itemize}
\end{proof}

\begin{thm}\label{genus_result}
For every $n\ge 3$, every surface triangulation $\Sigma$ with $n$ vertices and Euler genus $g$ has a connected dominating set $X$ of size at most $10n/21 + O(\sqrt{gn})$.
\end{thm}
\begin{proof}
  If $g = 0$, the result follows from \cref{main_result2}. If $1 \le g < n$, \cref{generic_surface_thm,main_result2} imply that $\Sigma$ has a connected dominating set $X$ of size at most $10n/21 + O(\sqrt{gn})$. If $g \ge n$, then $\sqrt{gn}\ge n$, so $V(G)$ is a connected dominating set of size $n \le 10n/21 + \sqrt{gn}$.
\end{proof}

\section{Applications in Graph Drawing}
\label{gd}
In this section, we apply our results to two problems in graph drawing: one-bend free sets and simultaneous embeddings with fixed edges.

\subsection{One-Bend Free Sets}

For a planar graph $G$, a set $Y\subseteq V(G)$ is called a \defin{free set} if, for every $|Y|$-point set $P\subseteq\R^2$, there exists a crossing-free straight-line drawing of $G$ in the plane such that the vertices of $Y$ are mapped to the points of $P$. A set $Y\subseteq V(G)$ is called a \defin{collinear set} if $G$ admits a crossing-free straight-line drawing in which all vertices of $Y$ lie on a single line. \citet{dujmovic.frati.ea:every} showed that a set is free if and only if it is collinear.
A \defin{$k$-bend drawing} of a graph $G$ is a crossing-free drawing in which each edge is represented by a polygonal chain consisting of at most $k+1$ line segments (with at most $k$ bends). A subset $Y\subseteq V(G)$ is a \defin{$k$-bend free set} if, for every $|Y|$-point set $P\subseteq\R^2$, $G$ admits a $k$-bend drawing in which the vertices of $Y$ are mapped to the points in $P$. A subset $Y\subseteq V(G)$ is a \defin{$k$-bend collinear set} if $G$ admits a $k$-bend drawing in which all vertices of $Y$ lie on a single line.
\citet{KaufmannWiese2002} showed that $V(G)$ is a $2$-bend free set for every planar graph $G$. This leaves open the question of $1$-bend free sets.
In this section, we prove:

\begin{thm}\label{one_bend_collinear_thm}
  For every $n\ge 3$, every $n$-vertex planar graph has a one-bend free set $L$ of size at least $11n/21=0.\overline{523809}n$. 
\end{thm}

Note that if the point set $P$ is contained in the $x$-axis, then in a $1$-bend drawing, no edge with both endpoints in $Y$ crosses the $x$-axis. Therefore, such a drawing yields a $2$-page book embedding of $G[Y]$. A triangulation has a $2$-page book embedding if and only if it is Hamiltonian \cite{bernhart.kainen:the}, in which case the cyclic ordering of vertices along the spine forms a Hamiltonian cycle.  The \defin{Goldner–Harary} graph $H$ is an $11$-vertex triangulation that is not Hamiltonian.  Let $G$ be the triangulation obtained by taking $k$ pairwise vertex-disjoint copies $H_1,\ldots,H_k$ of $H$ and greedily adding edges to get a triangulation.  If the graph $G$ has a $1$-bend collinear set $Y$ of size at least $10k+1$, then $V(H_i)\subseteq Y$ for at least one $i\in\{1,\ldots,k\}$.  Thus, $H_i$ has a $2$-page book embedding and therefore has a Hamiltonian cycle, which is not possible.  Therefore, the largest $1$-bend free set in $G$ is of size at most $10k = 10n/11=0.\overline{90}n$.

In the remainder of this subsection, we prove \cref{one_bend_collinear_thm}. We start by introducing the topological characterization of collinear sets from \cite{DBLP:journals/jocg/LozzoDFMR18}. A \defin{curve} $C$ is a continuous mapping from $[0, 1]$ to $\mathbb{R}^2$. We call $C(0)$ and $C(1)$ the \defin{endpoints} of $C$. If $C(0)=C(1)$, the curve is \defin{closed}; otherwise, it is \defin{open}. A curve $C$ is called \defin{simple} if $C(x) \neq C(y)$ for all $0\le x<y\le 1$, with the possible exception of $C(0)=C(1)$. A \defin{Jordan curve} is a simple closed curve.

Let $G$ be a plane graph. A Jordan curve $C$ is a \defin{$k$-proper good curve} if it contains a point in the interior of some face of $G$ (\defin{good}), and the intersection between $C$ and each edge $e$ of $G$ is empty, or at most $k$ points, or the entire edge $e$ (\defin{$k$-proper}).

\begin{thm}[\cite{DBLP:journals/jocg/LozzoDFMR18,dujmovic.frati.ea:every}]\label{straightline-topological}
Let $G$ be a plane graph and let $S\subseteq V(G)$. If there exists a $1$-proper good curve containing $S$, then $S$ is a collinear set, and hence a free set, of the underlying planar graph.
\end{thm}

The following observation, illustrated in \cref{subdivisions}, gives a similar sufficient condition for one-bend free sets.

\begin{figure}[htbp]
  \centering
  \includegraphics[page=7, scale=0.7]{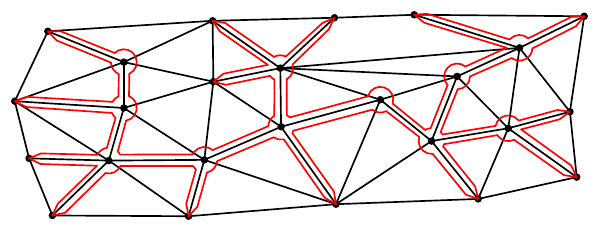} \\
  $\Downarrow$ \\
  \includegraphics[page=8, scale=0.7]{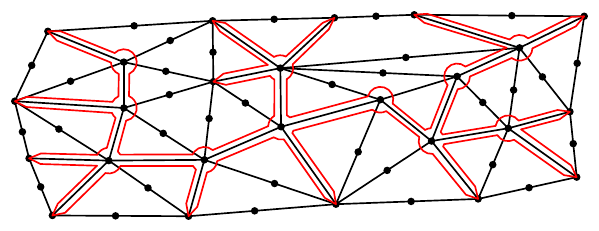}
  \caption{Subdividing $G$ so that a $2$-proper good curve $C$ becomes a $1$-proper good curve for the subdivided graph $G^+$.}
  \label{subdivisions}
\end{figure}

\begin{obs}\label{1-bend-topological}
Let $G$ be a plane graph. A set $S\subseteq V(G)$ is a one-bend free set if $G$ has a $2$-proper good curve containing $S$.
\end{obs}
\begin{proof}
Let $C$ be a $2$-proper good curve that contains $S$. For each edge $e \in E(G)$ such that $|C \cap e| = 2$, introduce a new subdivision vertex $u_e$ between the two intersection points of $C$ and $e$. By adding these new vertices, we obtain a plane drawing of a subdivision $G^+$ of $G$. Since every edge of $G^+$ is either contained in $C$ or intersects $C$ at most once, $C$ is a $1$-proper good curve for $G^+$ that contains $S$. Thus, by \cref{straightline-topological}, $S$ is a free set for $G^+$. For any prescribed $|S|$-point set $P$, drawing $G^+$ straight-line with $S$ on $P$ and suppressing the subdivision vertices produces a one-bend drawing of $G$. Therefore, $S$ is a one-bend free set for $G$.
\end{proof}

We next show that the leaves of a spanning tree of a planar graph form a one-bend free set.

\begin{lem}\label{spanning_tree_to_collinear_set}
  Let $G$ be a planar graph and let $T$ be a spanning tree of $G$. Then the leaves of $T$ form a one-bend free set for $G$.
\end{lem}

\begin{proof}
The result is immediate when $G$ has at most two vertices, so assume that $G$ has $n\ge 3$ vertices. Let $\Gamma$ be a straight-line plane drawing of $G$. Let $T$ be a spanning tree of $G$ that inherits the embedding $\Gamma$.  By \cref{1-bend-topological}, it suffices to construct a curve $\ell$ that is a $2$-proper good curve with respect to $\Gamma$ and that contains all the leaves of $T$.

For each vertex $v \in V(G)$, draw a small circle $C_v$ centred at $v$. Choose the radii of these circles small enough that for each $v \in V(G)$, $C_v$ intersects only edges incident to $v$ and is disjoint from all other vertex circles. Furthermore, for each edge $uv \in E(T)$ of the spanning tree $T$, draw two parallel segments on either side of $uv$ with endpoints on the boundaries of $C_u$ and $C_v$.  Choose these parallel segments sufficiently close to $uv$ that they do not intersect any other edges or segments (see \cref{tree_walking}).

\begin{figure}[htbp]
    \centering
    \includegraphics[page=3, scale=0.7]{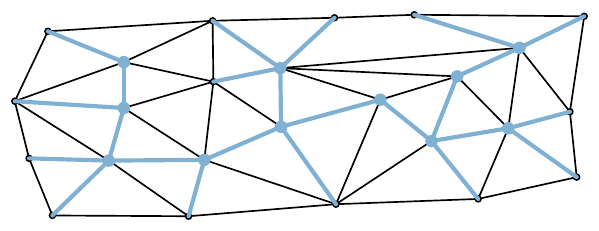}\\
    $\Downarrow$ \\
    \includegraphics[page=4, scale=0.7]{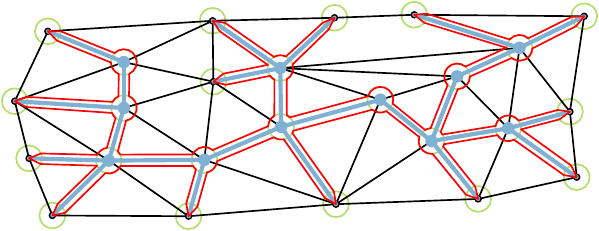}
    \caption{Constructing a $2$-proper good curve for $G$ that contains all the leaves of the tree $T$.}
    \label{tree_walking}
\end{figure}

Since $n\ge 3$, $T$ has at least one vertex of degree at least $2$, and no edge of $T$ joins two leaves. Root $T$ at an arbitrary vertex of degree at least $2$. We construct the curve $\ell$ by traversing $T$ in depth-first search order. For each directed tree edge $uv$ traversed, we add to $\ell$ the segment on the right side of the traversal direction. At each leaf $u$ with tree neighbour $v$, we join $u$ to the endpoints of the two segments along $uv$ on $C_u$. At each non-leaf vertex $u$, we connect the incident tree segments along circular arcs of $C_u$ in cyclic order.

The resulting curve $\ell$ is a closed Jordan curve containing all leaves of $T$ and enclosing all non-leaf vertices. Moreover, for each edge $uv \in E(G)$:
\begin{itemize}[nosep,nolistsep]
    \item If $uv \in E(T)$ and neither $u$ nor $v$ is a leaf, then $|uv \cap \ell| = 0$.
    \item If $uv \in E(T)$ and one endpoint (say $u$) is a leaf, then $uv\cap\ell = \{u\}$, so $|uv \cap \ell| = 1$.
    \item If $uv \in E(G) \setminus E(T)$, then $uv$ intersects $\ell$ once near each endpoint: at the circle boundary if the endpoint is an internal vertex of $T$, or at the vertex itself if the endpoint is a leaf. Thus $|uv \cap \ell| = 2$.
\end{itemize}
Therefore, $\ell$ is $2$-proper with respect to $\Gamma$. Since $\ell$ contains points in the interiors of faces of $\Gamma$, $\ell$ is a good curve. By \cref{1-bend-topological}, the leaves of $T$ form a one-bend free set for $G$.
\end{proof}

\begin{proof}[Proof of \cref{one_bend_collinear_thm}]
Let $G$ be an $n$-vertex planar graph with $n\ge 3$. Add edges to $G$ as needed to produce a triangulation $G'$ on the same vertex set $V(G)$. By \cref{main_result2}, $G'$ has a spanning tree $T'$ with at least $11n/21$ leaves. By \cref{spanning_tree_to_collinear_set}, the leaves of $T'$ form a one-bend free set of $G'$, which is also a one-bend free set of $G$.
\end{proof}

\subsection{Simultaneous Embedding with Fixed Edges and Without Mapping}

For an embedded planar graph $G$, a vertex set $L \subseteq V(G)$ is called \defin{cofacial} if all vertices in $L$ lie on a common face of $G[L]$. Motivated by simultaneous embedding with fixed edges and without mapping, \citet{angelini.evans.ea:sefe} prove the following result:

\begin{thm}[\cite{angelini.evans.ea:sefe}]
\label{induced_outerplane_orig}
    Every $n$-vertex plane graph $G$ (and therefore every triangulation) contains an induced outerplane graph $G[L]$ with at least $n/2$ vertices.
\end{thm}

Observe that if $G$ is a connected plane graph and $X$ is a connected dominating set, then the vertices and edges of $G[X]$ are contained in a single face $f$ of $G-X$. Domination implies that every remaining vertex is incident to $f$. Thus, $V(G)\setminus X$ is cofacial. If $X$ contains a vertex of the outer face of $G$, then $f$ is the outer face of $G-X$, so $G-X$ is an induced outerplane subgraph of $G$.
Combined with \cref{main_result2}, this implies the following improvement of \cref{induced_outerplane_orig}:

\begin{cor}\label{induced_outerplane_improved}
    Every connected $n$-vertex plane graph $G$ contains an induced outerplane graph $G'$ with at least $11n/21=0.\overline{523809}n$ vertices.
\end{cor}

\begin{proof}
The claim is immediate for $n\le3$. For $n\ge 4$, augment $G$ to a plane triangulation $G^+$ on the same vertex set, preserving the given embedding. Apply \cref{better_greedy} to obtain a connected dominating set $X$ of $G^+$ that contains a vertex of the outer face of $G^+$.  The observation above shows that $G^+-X$ is outerplane. Deleting the edges in $E(G^+)\setminus E(G)$ can only enlarge the outer face of $G^+-X$, so $G-X$ is outerplane as well.
\end{proof}

\Cref{induced_outerplane_improved} immediately implies an improved result for the graph drawing problem considered by \cite{angelini.evans.ea:sefe}, improving the bound from $n/2$ to $11n/21$:

\begin{thm}\label{sefenomap}
  For every $n$-vertex planar graph $G_1$ and every $\lceil 11n/21\rceil$-vertex planar graph $G_2$, there exist point sets $P_1\supseteq P_2$ with $|P_1|=n$, $|P_2|=\lceil 11n/21\rceil$ and crossing-free embeddings of $G_1$ and $G_2$ such that
  \begin{enumerate}[nosep,nolistsep,label=\rm(\alph*)]
    \item the vertices of $G_i$ are mapped to the points in $P_i$ for each $i\in\{1,2\}$;
    \item the edges of $G_2$ are drawn as line segments; and
    \item each edge $e$ of $G_1$ whose endpoints are both mapped to points in $P_2$ is drawn as a line segment.
  \end{enumerate}
\end{thm}

\begin{proof}[Proof of \cref{sefenomap}]
For $n\le2$, the claim is immediate. Otherwise, embed $G_1$ and augment it to a plane triangulation on the same vertex set.  By \cref{induced_outerplane_improved}, there exists $Y\subseteq V(G_1)$ of size $m:=\lceil11n/21\rceil$ such that $G_1[Y]$ is outerplane.

Fix some straight-line non-crossing drawing of $G_2$ so that no three vertices of $G_2$ are collinear. Let $P_2$ be the (general position) point set on which the vertices of $G_2$ are drawn.  Since $G_1[Y]$ is outerplane, it can be drawn straight-line on $P_2$ while preserving its plane embedding, with a suitable bijection of its vertices to the points \cite{angelini.evans.ea:sefe}. Extend this drawing of $G_1[Y]$ to a topological plane drawing of $G_1$: an embedding-preserving homeomorphism between the original and new drawings of the induced subgraph extends over its faces, and maps the remaining vertices and edges of $G_1$ into the interior of the outer face of $G_1[Y]$ without additional crossings. Let $P_1$ be the vertex set of this drawing. Then $P_1\supseteq P_2$, $|P_1|=n$, $|P_2|=m$, and the drawings of $G_1$ and $G_2$ satisfy the three conditions of the theorem.
\end{proof}

We end this section with the following observation. The complement $Y:=V(G)\setminus X$ of a connected dominating set $X$ is cofacial, as argued above.  Here, we show that, in triangulations, the smallest connected dominating set problem and the maximum cofacial set problem are duals of each other.

\begin{thm}\label{cofacial_equivalence}
  Let $G$ be a triangulation with $n\ge4$ vertices, let $X$ be a minimum-sized connected dominating set of $G$, and let $Y$ be a maximum-sized cofacial set in $G$. Then $|X|+|Y|=n$.
\end{thm}

\begin{proof}
  Since the complement of any connected dominating set of $G$ is cofacial, $|Y| \ge n - |X|$, which gives $|X| + |Y| \ge n$. To prove the reverse inequality, it suffices to construct a connected dominating set of size $n - |Y|$.

  Let $Y$ be a maximum-sized cofacial set in $G$ and let $X:=V(G)\setminus Y$.  For each face $f$ of $G[Y]$, let $X_f\subseteq X$ be the subset of $X$ in the interior of $f$.  Let $f_0$ be a face of $G[Y]$ that has all the vertices of $Y$ on its boundary. We claim that $Y$ can be chosen such that $X_{f_0}=X$. Choose the outer face of $G$ so that this face is contained in $f_0$.  Suppose that $X_{f_0}\neq X$ and consider the generalized near-triangulation $G_Y:=G-X_{f_0}$.  Since $X_{f_0}\neq X$, $I(G_Y)$ is non-empty, so there exists a vertex $v\in B(G_Y)$ with $|N^+_{G_Y}(v)|\ge 1$. Let $f_0'$ be the face of $G-(X_{f_0}\cup\{v\})$ that contains $f_0$ (and $v$) in its interior.  Let $Y':=(Y\cup N^+_{G_Y}(v))\setminus \{v\}$ and let $G_{Y'}:=G-(X_{f_0}\cup\{v\})$.  Observe that $B(G_{Y'})=Y'$, so $Y'$ is a cofacial set and $|Y'| = |Y| + |N^+_{G_Y}(v)| - 1\ge |Y|$.  Since $Y$ is a maximum-sized cofacial set, we must have $|Y'|=|Y|$ and $|N^+_{G_Y}(v)|=1$.  Furthermore, $f_0'$ contains $f_0$ and $v$ in its interior, so the subset of $V(G)\setminus Y'$ in the interior of $f_0'$ contains $X_{f_0}\cup\{v\}$, which has size $|X_{f_0}|+1>|X_{f_0}|$. Thus, we have found a new cofacial set $Y'$ with the same size as $Y$, such that the face $f_0'$ that has all vertices of $Y'$ on its boundary contains $f_0$ and strictly more vertices of its complement than $f_0$. We can repeat this operation until we eventually obtain a cofacial set $Y^*$ with $|Y^*|=|Y|$ such that the outer face $f_0^*$ of $G[Y^*]$ has all vertices of $X^*:=V(G)\setminus Y^*$ in its interior.

  We now assume that $Y$ is chosen so that $X_{f_0}=X$.  All that remains is to show that $X$ is a connected dominating set of $G$.  Since each face of $G$ has only three vertices on its boundary and $n\ge 4$, $Y\neq V(G)$, so $X$ is non-empty.  First, we argue that $G[X]$ is connected. For any two distinct vertices $v,w\in X$, let $\gamma$ be a curve with endpoints $v$ and $w$ contained in the interior of $f_0$ whose interior avoids all vertices of $G$ and crosses edges of $G$ transversely. Let $t_1,\ldots,t_k$ be the triangular faces of $G$ whose interiors are intersected by $\gamma$, in the order they occur while traversing $\gamma$. If $k=1$, then $v$ and $w$ belong to the same triangle $t_1$, so $vw\in E(G[X])$. Otherwise, for each $i\in\{1,\ldots,k-1\}$, the triangles $t_i$ and $t_{i+1}$ share an edge $e_i$. Because $\gamma$ lies in the interior of $f_0$, $e_i$ cannot be an edge of $G[Y]$, so at least one endpoint $v_i$ of $e_i$ belongs to $X$. For each $i\in\{1,\ldots,k-2\}$, the vertices $v_i$ and $v_{i+1}$ both belong to the triangle $t_{i+1}$, so $v_i=v_{i+1}$ or $v_i v_{i+1}\in E(G[X])$. Furthermore, $v, v_1 \in V(t_1) \cap X$, so $v=v_1$ or $vv_1\in E(G[X])$, and similarly $w=v_{k-1}$ or $wv_{k-1}\in E(G[X])$. Thus $G[X]$ contains a walk from $v$ to $w$, showing that $G[X]$ is connected.

  To see that $X$ dominates $G$, suppose that some $u\in Y$ is not adjacent to any vertex in $X$. Then $N_G(u)\subseteq Y$, so all triangular faces of $G$ incident to $u$ have all their vertices and edges in $G[Y]$. None of these faces contains a vertex of $G$ in its interior, so none of these faces is $f_0$ (since $f_0$ contains $X$ in its interior). The union of these triangles is a disc with $u$ in its interior. This contradicts the fact that $f_0$ has all vertices of $Y$ on its boundary.

  Therefore, $X$ is a connected dominating set of $G$, so $|X| \le |V(G) \setminus Y| = n - |Y|$. Combining this with $|X| + |Y| \ge n$ yields $|X| + |Y| = n$.
\end{proof}

\section{Conclusions}
\label{conclusions}
In this paper, we improve the long-standing upper bound of $n/2$ for the size of connected dominating sets in triangulations, showing that every $n$-vertex triangulation admits a connected dominating set of size at most $10n/21$. This result narrows the gap to the best-known lower bound and has applications to graph drawing.  Though not the focus of this paper, the proof is constructive and leads to an algorithm that can find such a set in $O(n)$ time. A description of the (simple) data structure required for a linear-time implementation appears in the preprint \href{https://arxiv.org/abs/2312.03399v2}{https://arxiv.org/abs/2312.03399v2}.  The same is true for the extension of this algorithm to triangulations on surfaces, which relies only \cref{Eppstein} and the planar result.

Two natural open questions arise from this work:

\begin{enumerate}
  \item What is the infimum of the constants $c$ for which every $n$-vertex triangulation (with $n\ge 3$) has a connected dominating set of size at most $c n+O(1)$? Our main result, \cref{main_result2}, and the recent lower bound \cite{enami2025contributionsconjecturesplanargraphs} place this infimum in $[7/18,10/21]$.

  \item What is the supremum of the constants $\alpha$ for which every $n$-vertex planar graph contains a one-bend free set of size at least $\alpha n-O(1)$?  \cref{one_bend_collinear_thm} gives a lower bound of $11/21$ on this supremum, and disjoint copies of the Goldner–Harary graph give an upper bound of $10/11$.

\end{enumerate}

\section*{History and Statement on AI Use}
The results in this paper were obtained in 2023. Vida~Dujmović presented the results at the \href{https://math.nyu.edu/seminars/geometry/fall23/Dujmovic.html}{New York Geometry Seminar} on December~5, 2023, and Pat~Morin presented these results at the Dagstuhl Seminar 24062 on Beyond-Planar Graphs on February~5, 2024. The first preprint, \href{https://arxiv.org/abs/2312.03399v1}{arXiv:2312.03399v1}, includes \cref{main_result2,one_bend_collinear_thm} and was posted on December~6, 2023. An updated preprint, \href{https://arxiv.org/abs/2312.03399v2}{arXiv:2312.03399v2}, posted in April~2024, includes \cref{generic_surface_thm}.  A preliminary version of this article appeared in the proceedings of \href{https://doi.org/10.4230/LIPIcs.ICALP.2026.41}{ICALP~2026}.

Before journal submission, the authors asked GPT-6 Astra for a review of the manuscript for errors and inconsistencies. It proofread the paper, pointed out several minor errors and omissions, suggested easy fixes, and offered suggestions for simplifications and improvements to the exposition.  The current version of the paper is the result of incorporating some of these suggestions.  The most significant of these is the use of potential functions rather than linear programs to bound the sizes of the connected dominating sets in \cref{simple_greedy,better_greedy}. One of the authors (Pat Morin) uses an AI-integrated text editor (Antigravity IDE with Gemini) while working on the manuscript.

\bibliographystyle{plainurlnat}

\bibliography{main-journal}
\end{document}